\documentclass[12pt, a4paper]{amsart}

\usepackage[hmargin=30mm, vmargin=25mm, includefoot, twoside]{geometry}
\usepackage[bookmarksopen=true]{hyperref}

\usepackage{amsfonts,amssymb,verbatim}

\usepackage{latexsym}
\usepackage{mathrsfs}
\usepackage{stmaryrd}
\usepackage{xspace}
\usepackage{enumerate, paralist}
\usepackage{graphicx}
\usepackage[all]{xy}
\usepackage{extarrows}
\usepackage{tabu}
\usepackage{accents}
\usepackage{tensor}
\usepackage{enumitem}

\usepackage[usenames,dvipsnames]{color}

\usepackage{txfonts, pxfonts}

\usepackage{amsmath}

\usepackage[tbtags]{mathtools}

\usepackage{amsthm, thmtools}

\newtheorem{thm}{Theorem}[section]
 \newtheorem{cor}[thm]{Corollary}
 \newtheorem{lem}[thm]{Lemma}
 \newtheorem{prop}[thm]{Proposition}

 \newtheorem{conv}[thm]{Convention}
\newtheorem{introthm}{Theorem}

 \newtheorem{introdefn}[introthm]{Definition}

\numberwithin{equation}{section}

 \theoremstyle{definition}
  \newtheorem{defn}[thm]{Definition}

 \theoremstyle{remark}
 \newtheorem{rem}[thm]{Remark}

\newtheorem*{claim*}{Claim}

\def\NN{\mathbb{N}}
\def\CC{\mathbb{C}}
\def\RR{\mathbb{R}}

\def\B{\mathfrak{B}}

\def\KK{\mathfrak{K}}

\def\H{\mathcal{H}}
\def\K{\mathcal{K}}
\def\N{\mathcal{N}}

\def\Pz{\mathscr{P}_0}

\def\SOT{\mathrm{(SOT)}}
\def\Id{\mathrm{Id}}

\def\Tr{\mathrm{Tr}}

\def\ppg{\mathrm{prop}}
\def\ran{\mathrm{Ran}}

\def\Ad{\mathrm{Ad}}

\def\supp{\mathrm{supp}}
\def\diam{\mathrm{diam}}
\def\d{\mathrm{d}}

\def\RR{\mathrm{RR}}
\def\SP{\mathrm{(SP)}}

\def\gr{\mathrm{graph}}

\begin{document}

\title[Characterisations for coarse embeddability]{Characterisations for coarse embeddability via quasi-locality and Schur multipliers}
\author{Kang Li}
\author{Jiawen Zhang}
\author{Jingming Zhu}

\address[Kang Li]{Centre for Mathematical Sciences, Lund University, Box 118, SE-221 00 Lund, Sweden.}
\email{kang.li@math.lth.se}

\address[Jiawen Zhang]{School of Mathematical Sciences, Fudan University, 220 Handan Road, Shanghai, 200433, China.}
\email{jiawenzhang@fudan.edu.cn}

\address[Jingming Zhu]{College of Data Science, Jiaxing University, 899 Guangqiong Road, Jiaxing, 314000, China.}
\email{jingmingzhu@zjxu.edu.cn}

\thanks{Jiawen Zhang was partly supported by the National Key R{\&}D Program of China 2022YFA100700 and NSFC 12422107.}



\begin{abstract}
We establish three new characterisations of coarse embeddability into Hilbert space for discrete metric spaces of bounded geometry in terms of escaping quasi-local couplings, Schur multipliers, and operator norm state localisation. As applications, we prove that if a sparse metric space of bounded geometry is coarsely embeddable into Hilbert space, then every ghost projection in the quasi-local algebra is compact, while we also provide a counterexample to the converse implication. Finally, we characterise coarse embeddability into Hilbert space via the relative Haagerup property for the quasi-local algebra, and investigate its relation to coarse embeddability into $\ell^p$ for $p>2$.
\end{abstract}

\date{\today}

\maketitle

\section{Introduction}\label{sec:intro}

The notion of coarse embedding was introduced by Gromov in \cite{Gro93}, providing a natural way to compare the coarse geometry of metric spaces. It can be viewed as a coarse analogue of isometric embedding, allowing controlled distortion of metrics. A map $f: (X,d_X) \to (Y,d_Y)$ between metric spaces is called a \emph{coarse embedding} if there exist non-decreasing functions $\rho_-, \rho_+: [0,+\infty) \to [0,+\infty)$ with $\lim\limits_{t\to +\infty} \rho_{\pm}(t) = +\infty$ such that
\[
\rho_-(d_X(x,y)) \leq d_Y(f(x), f(y)) \leq \rho_+(d_X(x,y)) \quad \text{for any} \quad x,y\in X.
\]
Among various choices of target spaces, Hilbert space is of particular interest. Indeed, Gromov suggested that a finitely generated group admitting a coarse embedding into Hilbert space should satisfy the Novikov conjecture \cite{Gro93}. This connection illustrates the significance of coarse embeddability beyond metric geometry, particularly in topology, geometry and operator algebras. A major breakthrough in this direction was subsequently achieved by Yu in \cite{Yu00}, where he showed that the coarse Baum-Connes conjecture holds for metric spaces of bounded geometry which are coarsely embeddable into Hilbert space.

The coarse Baum-Connes conjecture is a central problem in higher index theory \cite{HR95}, having significant applications in topology and geometry (see, \emph{e.g.}, \cite{Roe93, Roe96}). It asserts that the \(K\)-theory of Roe algebras can be computed in geometric terms. Roe algebras are \(C^*\)-algebras introduced by J. Roe in his study of elliptic differential operators on open manifolds \cite{Roe88}, where he showed that their \(K\)-theory serves as a receptacle for higher indices. The quasi-local algebra also plays an important role in higher index theory. Indeed, Engel showed in \cite{Eng14} (see also \cite{Eng18}) that the higher indices of elliptic pseudodifferential operators on open manifolds naturally take values in the \(K\)-theory of quasi-local algebras.
%However, it remains a long-standing open question whether the Roe and quasi-local algebras have the same \(K\)-theory.

%Nevertheless, the relationship between Roe algebras and quasi-local algebras, particularly at the level of \(K\)-theory, remains far from being fully understood.

However, the coarse Baum-Connes conjecture does not hold in general. The first counterexample was discovered in \cite{HLS02} using expander graphs, where non-compact ghost projections play a crucial role. Roughly speaking, an operator is called a ghost (due to Yu) if its matrix coefficients vanish at infinity. For suitable expander graphs, non-compact ghost  projections in the Roe algebra are the main obstruction to the coarse Baum-Connes conjecture.
%give rise to \(K\)-theory classes which do not lie in the image of the coarse assembly map, thereby violating the surjectivity.

%The preceding counterexamples suggest a close connection between coarse embeddability and ghost operators. Indeed, for sparse metric spaces admitting a coarse embedding into Hilbert space, it is known that every ghost projection in the Roe algebra is compact, as a consequence of \(K\)-theoretic results. It is natural to ask whether this conclusion extends to the larger quasi-local algebra. However, since the relationship between the \(K\)-theories of Roe and quasi-local algebras remains unclear in general, the existing \(K\)-theoretic arguments do not directly apply to this setting. This motivates us to seek a direct operator-theoretic approach to ghost compactness, and more generally, to investigate whether the behaviour of ghost operators can be used to characterise coarse embeddability itself.

This leads naturally to the relationship between coarse embeddability and ghost operators. For sparse\footnote{Recall that a metric space is \emph{sparse} if it can be written as a coarse disjoint union of a sequence of finite subspaces.} metric spaces which coarsely embed into Hilbert space, it is known that every ghost projection in the Roe algebra is compact, based on $K$-theoretic results \cite{Fin14,GLW26,WFZ25}. It is natural to ask whether this can be extended to ghost projections in the quasi-local algebra. Since the relationship between the \(K\)-theories of Roe and quasi-local algebras remains unclear in general (see \cite{LZZ26}), the existing \(K\)-theoretic arguments do not directly apply. This motivates us to seek an operator-theoretic approach and further investigate whether the behaviour of ghost operators can characterise coarse embeddability itself.

In this paper, we establish three new characterisations of coarse embeddability into Hilbert space and provide several  applications, especially to ghost operators. In the following, we will explain them in turn.

To fix the notation, let \(X\) be a discrete metric space of bounded geometry. We write \(C_u^*(X)\) and \(C_{uq}^*(X)\) for its uniform Roe algebra and uniform quasi-local algebra, respectively, and \(\KK(\ell^2(X))\) for the algebra of compact operators. The precise definitions are recalled in Section~\ref{sec:pre}.

\subsection{Escaping quasi-local couplings and ghost projections}

Our first characterisation is inspired by the notion of generalised expanders from \cite{Tes09} and the work on quasi-locality for ghost projections \cite{KLVZ21,LSZ23}. Replacing the Poincar\'{e} inequality therein by a spatial version of quasi-locality (see Definition~\ref{defn:ppg and quasi-local} for the classical operator version of quasi-locality), which is formally easier to verify, we introduce:

\begin{introdefn}\label{defn:escaping quasi-local coupling}
Let $(X,d)$ be a discrete metric space of bounded geometry. An \emph{escaping quasi-local coupling} on $X$ is a sequence of symmetric finitely supported probability measures $\{\nu_n\}_n$ on $X\times X$ satisfying:
\begin{description}[
  font=\normalfont\bfseries,
  leftmargin=3em,
  labelsep=0.5em,
  style=multiline
]
\item[\textbf{(E)}] For any $S>0$, $\nu_n\bigl(\{(x,y)\in X^2:d(x,y)\le S\}\bigr)\to 0$;
\item[\textbf{(Q)}] For any $\varepsilon>0$, there is $R>0$ such that $\nu_n(A\times B)<\varepsilon$ for any $n$ and $A,B \subseteq X$ with $d(A,B)>R$.
\end{description}
\end{introdefn}

Here \textbf{(E)} stands for ``escaping'', which means that the mass of the measures \(\{\nu_n\}_n\) on each entourage tends to zero. On the other hand, \textbf{(Q)} stands for ``quasi-locality'', which indicates a spatial version of quasi-locality.

We show that the existence of such couplings precisely characterises the failure of coarse embeddability into Hilbert space:

\begin{introthm}\label{thm:esc and ql coupling}
A discrete metric space $(X,d)$ of bounded geometry can be coarsely embedded into Hilbert space \emph{if and only if} $X$ admits no escaping quasi-local coupling.
\end{introthm}

As an application of Theorem~\ref{thm:esc and ql coupling}, we generalise the known compactness result for ghost projections in Roe algebras to the quasi-local setting (see Theorem~\ref{prop:ghost proj are compact}).

\begin{introthm}\label{introcor:ghost proj are compact}
Let $(X,d)$ be a sparse discrete metric space of bounded geometry which is coarsely embeddable into Hilbert space. Then for every ghost projection \(P\in\B(\ell^2(X))\),
\[
\operatorname{dist}\big(P,C_{uq}^*(X)) = \operatorname{dist}\big(P,C_u^*(X)\big)
=
\begin{cases}
0,& \text{if }P\text{ is compact};\\
\frac12,&\text{if }P\text{ is non-compact}.
\end{cases}
\]
In particular, any ghost projection in $C^*_{uq}(X)$ is compact.
\end{introthm}

The same conclusion holds for quasi-local algebras with coefficients in any separable Hilbert space. Our proof does not rely on \(K\)-theoretic arguments, but instead shows that a non-compact ghost projection would give rise to an escaping quasi-local coupling, contradicting Theorem~\ref{thm:esc and ql coupling}.

We show that the converse of Theorem~\ref{introcor:ghost proj are compact} does not hold in general. More precisely, using a construction of Delabie and Khukhro from \cite{DK18}, we obtain a sparse metric space of bounded geometry which does not coarsely embed into Hilbert space, while every quasi-local ghost projection is compact (see Proposition~\ref{prop:ghost proj are compact converse}).

\subsection{Schur multipliers and ghost operators}

As seen above, for sparse metric spaces, the compactness of all ghost \emph{projections} in the quasi-local algebra is a necessary but not sufficient condition for coarse embeddability into Hilbert space. On the other hand, Roe and Willett showed in \cite{RW14} that every ghost \emph{operator} in the uniform Roe algebra is compact if and only if the underlying space has Property A from \cite{Yu00}. Property A is strictly stronger than coarse embeddability, as shown in \cite{AGS12}. These results motivate us to seek an appropriate class of ghost operators whose compactness characterises coarse embeddability into Hilbert space.

Our second characterisation provides such a condition using Schur multipliers. For a set $X$, recall that a Schur multiplier $M_\varphi$ acts on $\B(\ell^2(X))$ by multiplying the matrix coefficients of operators by a kernel $\varphi: X \times X \to \CC$, called a Schur multiplier kernel. For a metric space \((X,d)\), denote by $\Pz(X)$ the set of all Schur multiplier kernels which are normalised positive-definite and $\lim\limits_{R\to\infty}\sup_{d(x,y)>R}|\varphi(x,y)|=0$. An empty supremum of non-negative quantities is taken to be zero. We introduce:
\[
J_0(X):=\{T \in C^*_u(X):M_\varphi(T^*T)\in\KK(\ell^2(X)) \text{ for every } \varphi\in\Pz(X)\}.
\]
We show that \(J_0(X)\) is an ideal lying between the ideal of compact operators and the ideal of ghost operators. The following result allows us to detect the obstruction to coarse embeddability with the help of $J_0(X)$ rather than merely ghost projections.

%More importantly, its compactness characterises coarse embeddability into Hilbert space.

\begin{introthm}\label{thm:char by multiplier}
Let $(X,d)$ be a discrete metric space of bounded geometry. Then the following are equivalent:
\begin{enumerate}
\item The space $X$ can be coarsely embedded into some Hilbert space;
\item There exists a sequence $\{\varphi_n\}_{n\in\mathbb N}\subseteq \Pz(X)$ such that $\|M_{\varphi_n}(T)-T\| \to 0$ as $n\to \infty$ for any $T \in C^*_{uq}(X)$;
\item Any positive $T\in C^*_u(X)$ with $M_\varphi(T)\in \KK(\ell^2(X))$ for any $\varphi\in\Pz(X)$ is compact;
\item Any positive $T\in C^*_{uq}(X)$ with $M_\varphi(T)\in \KK(\ell^2(X))$ for any $\varphi\in\Pz(X)$ is compact;
\item $J_0(X)=\KK(\ell^2(X))$.
\end{enumerate}
\end{introthm}

%Theorem~\ref{thm:char by multiplier} explains how the compactness condition for ghost projections can be strengthened to obtain a characterisation of coarse embeddability. Although non-compact ghost projections need not exist when coarse embeddability fails, there is always a non-compact ghost operator \(T \in C^*_u(X)\) for which \(M_\varphi(T^*T)\) is compact for every \(\varphi\in\mathcal P_0(X)\). Thus, Schur multipliers allow us to detect the obstruction to coarse embeddability that may be invisible to ghost projections.

As an application, we obtain a characterisation of coarse embeddability in terms of the relative Haagerup property. More precisely, we show that \(X\) coarsely embeds into Hilbert space \emph{if and only if} \(C_{uq}^*(X)\) has the \(E\)-Haagerup property with respect to its diagonal subalgebra \(\ell^\infty(X)\), where \(E\) denotes the canonical conditional expectation. This extends the known result in \cite{KLS22} from uniform Roe algebras to uniform quasi-local algebras (see Proposition~\ref{thm:relative_Haagerup}).

As another application of Theorem~\ref{thm:char by multiplier}, we show that for a fixed $p\in (2,+\infty)$, a discrete metric space \(X\) of bounded geometry can be coarsely embedded into Hilbert space \emph{if and only if} it can be coarsely embedded into \(\ell^p\) together with either \(J_0(X)\) having real rank zero or \(J_0(X)/\KK(\ell^2(X))\) having property $\SP$ (see Proposition~\ref{thm:lp and l2}). This provides a partial answer to the open question whether coarse embeddability into \(\ell^p\) implies coarse embeddability into Hilbert space for a discrete metric space of bounded geometry (see \cite{JR06}).

%We show that, for a discrete metric space \(X\) of bounded geometry which coarsely embeds into \(\ell^p\), Hilbert-space coarse embeddability is equivalent to \(J_0(X)\) having real rank zero, or equivalently, to the quotient \(J_0(X)/\mathcal K(\ell^2(X))\) having property (SP). The proof combines our Schur-multiplier characterisation with the compactness of ghost projections on sparse spaces coarsely embeddable into \(\ell^p\) (see Theorem 6.7).

\subsection{Operator norm state localisation}

The distinction between Property A and coarse embeddability into Hilbert space discussed above can also be seen from the viewpoint of operator norm localisation. Recall that the operator norm localisation property (ONL) from \cite{CTWY08} asserts, roughly speaking, that the norm of a positive finite-propagation operator can be almost attained by a unit vector with support having controlled  diameter. Sako proved in \cite{Sako14} that ONL is equivalent to Property A, and hence it is natural to ask whether an appropriate weakening of ONL could characterise coarse embeddability into Hilbert space.

To this end, we replace localised vector states by general normal states. Given a positive trace-class operator $D$ on $\ell^2(X)$ with $\Tr(D)=1$, we have a normal state on $\B(\ell^2(X))$ given by $T \mapsto \Tr(TD)$. It is known that every normal state $\phi$ on $\B(\ell^2(X))$ has this form, and $D$ is called a \emph{density operator} of $\phi$.
Instead of requiring the supporting vectors to be localised in bounded regions in the case of ONL, we impose a decay condition on the matrix coefficients of the corresponding density operators. This leads us to introduce the following notion:

\begin{introdefn}\label{defn:ONSL}
A discrete metric space $(X,d)$ of bounded geometry is said to have the \emph{operator norm state localisation property (ONSL)} if for any $R>0$ and $\varepsilon \in (0,1)$, there exists a non-increasing function $\omega:[0,\infty)\to [0,1]$ with $\lim_{s\to\infty}\omega(s)=0$ such that for any positive operator $T\in \B(\ell^2(X))$ with $\ppg(T)\le R$, there exists a positive trace-class operator $D$ on $\ell^2(X)$ with $\Tr(D)=1$ satisfying $\Tr(TD)\ge(1-\varepsilon)\|T\|$ and $|D_{xy}|\le \omega(d(x,y))\sqrt{D_{xx}D_{yy}}$ for any $x,y\in X$.
\end{introdefn}

We show that ONSL provides another characterisation of coarse embeddability.

\begin{introthm}\label{thm:CE and ONL}
Let $(X,d)$ be a discrete metric space of bounded geometry. Then $X$ can be coarsely embedded into Hilbert space \emph{if and only if} $X$ satisfies ONSL.
\end{introthm}

On the other hand, Property A can be recovered by restricting ranks of density operators. More precisely, we say that \(X\) has \emph{ONSL with uniformly finite rank} if, for each \(R>0\) and \(\varepsilon\in(0,1)\), the density operators \(D\) in Definition~\ref{defn:ONSL} can be chosen with rank bounded by a constant depending only on \(R\) and \(\varepsilon\). We manage to show that a discrete metric space of bounded geometry has Property A \emph{if and only if} it has ONSL with uniformly finite rank (see Proposition~\ref{prop:ONL and bounded rank} and Corollary~\ref{cor:ONSL with bdd rank}).

%An interesting feature of ONSL is that the density operators $D$ in Definition~\ref{defn:ONSL} are not required to have uniformly bounded rank. In fact, imposing such a restriction recovers the original ONL property. More precisely, we say that \(X\) has \emph{ONSL with uniformly finite rank} if, for each \(R>0\) and \(\varepsilon\in(0,1)\), the density operators \(D\) in Definition E can be chosen with rank bounded by a constant depending only on \(R\) and \(\varepsilon\). We prove that

%We prove that ONSL with uniformly finite rank is equivalent to ONL and hence, property A by Sako's result. Thus, the distinction between Property A and coarse embeddability into Hilbert space can be detected by the ranks of the density operators used in ONSL.

%\begin{introprop}\label{prop:A iff ONSL unif}
%A discrete metric space of bounded geometry has property A \emph{if and only if} it has ONSL with uniformly finite rank.
%\end{introprop}

\subsection{Ideas of the proofs}

The proofs of our three characterisations rely on two key ingredients. The first comes from Tessera's characterisation of coarse embeddability in terms of generalised expanders \cite{Tes09}, which provides geometric obstructions through certain Poincar\'{e} inequalities. In our first characterisation, these inequalities lead directly to escaping quasi-local couplings. For the other two characterisations, we develop this approach further by translating the Poincar\'{e} inequalities into suitable operator inequalities. This allows us to construct operators with non-trivial spectral information, controlled by matrices whose supports escape to infinity (see Proposition~\ref{prop:matrix-obstruction}).

The second is a quantitative Schur multiplier estimate based on a finite cut decomposition of the Hilbert metric. Given a coarse map \(F:X\to \H\) from a metric space $X$ to a Hilbert space $\H$, consider the kernel $\phi_t(x,y)=e^{-t\|F(x)-F(y)\|}$ for $t>0$, which is positive-definite. We establish an estimate that relates the Schur multiplier \(M_{\phi_t}(T)\) to the quasi-local behaviour of \(T\). In particular, we show that \(M_{\phi_t}(T) \to T\) as $t \to 0$ for any quasi-local $T$ (see Proposition~\ref{lem:ql-approx} and Corollary~\ref{cor:ql-approx}).

\subsection*{Organisation} Section~\ref{sec:pre} contains the necessary preliminaries. In Sections~\ref{sec:obstructions} and \ref{sec:cuts}, we develop the two main tools described above: operator-theoretic obstructions arising from generalised expanders and Schur multiplier estimates based on cut decompositions. Sections~\ref{sec:char I}, \ref{sec:schur-ghost} and \ref{sec:onsl} are devoted to the proofs of our three characterisations, respectively, together with their applications.

\subsection*{AI use statement} We acknowledge the use of OpenAI's ChatGPT in exploring mathematical ideas related to some of the results in this paper and in improving parts of the exposition. All mathematical arguments were carefully checked and revised by the authors, who bear full responsibility for the final manuscript.

\section{Preliminaries}\label{sec:pre}

For a discrete metric space $(X,d)$, $x\in X$ and $R \geq 0$, denote the closed ball $B(x,R):=\{y\in X:d(x,y) \leq R\}$, or $B_X(x,R)$ if we need to emphasise the space.
We say that $(X,d)$ has \emph{bounded geometry} if for every $R\geq 0$, we have
\[
N_X(R):=\sup_{x\in X} |B(x,R)| < +\infty.
\]
A discrete metric space is called \emph{locally finite} if every bounded subset is finite.
For $R \geq 0$, denote the \emph{$R$-neighbourhood} of $A \subseteq X$ by $\N_R(A):=\{x\in X: d(x,A) \leq R\}$.
For $R\geq 0$, denote the \emph{$R$-entourage} $E_R:=\{(x,y): d(x,y) \leq R\}$.
%
%A subset $K \subseteq X \times X$ is called a \emph{partial translation} if there exists $S>0$ such that $K \subseteq E_S$ and for each $x\in X$, there is at most one $y\in X$ with $(x,y) \in K$ and at most one $z\in X$ with $(z,x) \in K$. In this case, there is a bijection $t: D \to R$ with $D,R \subseteq X$ such that $K$ coincides with $\operatorname{graph}(t)=\{(t(y),y):y\in D\}$.

We say that a metric space $(X,d)$ can be \emph{coarsely embedded into Hilbert space} if there exist a map $F: X\to \H$ for some Hilbert space $\H$ and non-decreasing functions $\rho_-, \rho_+: [0,+\infty) \to [0,+\infty)$ (called \emph{control functions}) such that $\lim\limits_{t\to +\infty} \rho_{\pm}(t) = +\infty$ and
\[
\rho_-(d(x,y)) \leq \|F(x) - F(y)\|_\H \leq \rho_+(d(x,y)) \quad \text{for any} \quad x,y\in X.
\]
Say that $f:(X,d_X) \to (Y,d_Y)$ between two metric spaces is \emph{coarse} if there is a non-decreasing function $\rho_+: [0,+\infty) \to [0,+\infty)$ such that $d_Y(f(x), f(y)) \leq \rho_+(d_X(x,y))$ for any $x,y\in X$ and $f^{-1}(B)$ is bounded for any bounded subset $B\subset Y$.

Given a Hilbert space $\H_0$, denote the Hilbert space consisting of all square summable functions from $X$ to $\H_0$ by $\ell^2(X;\H_0) \cong \ell^2(X) \otimes \H_0$. Denote all bounded linear operators on this Hilbert space by $\B(\ell^2(X;\H_0))$, and all compact operators by $\KK(\ell^2(X;\H_0))$. For $T \in \B(\ell^2(X;\H_0))$ and $x,y \in X$, define $T_{xy} \in \B(\H_0)$ by
\[
\langle T_{xy}\xi,\eta \rangle = \langle T(\delta_y \otimes \xi), \delta_x \otimes \eta \rangle \quad \text{for} \quad \xi, \eta \in \H_0.
\]
Here $\delta_x$ denotes the Dirac function at $x$ in $\ell^2(X)$. The \emph{support} of $T$ is defined to be
\[
\supp(T):=\{(x,y) \in X \times X: T_{xy} \neq 0\}.
\]
Given $A \subseteq X$, denote by $\chi_A$ the characteristic function of $A$, also regarded as a multiplication operator on $\ell^2(X;\H_0)$.

For two metric spaces $X,Y$ and $T \in \B(\ell^2(X;\H_0), \ell^2(Y;\H_0))$, we define $\supp(T)$ similarly. Given a map $f: X\to Y$ and a Hilbert space $\H_0$, say that an isometry $V: \ell^2(X;\H_0) \to \ell^2(Y;\H_0)$ \emph{covers $f$} if there is $C>0$ such that $\supp(V)$ is contained in $\{(y,x) \in Y \times X: d_Y(f(x),y) \leq C\}$. Denote $\Ad_V(T):= VTV^*$ for $T \in \B(\ell^2(X;\H_0))$.

%For a closed subspace $U\subseteq \ell^2(X;\H_0)$, denote by $P_U \in \B(\ell^2(X;\H_0))$ the orthogonal projection onto $U$. Denote $u_A:=\chi_A u$ for $u\in \ell^2(X;\H_0)$ and $A \subseteq X$, and $P_u:=P_{\CC u}$ if $\|u\|=1$. Here $\chi_A$ denotes the characteristic function of $A$.

A sequence of finite metric spaces $\{(X_n,d_n)\}_{n\in \NN}$ is said to have \emph{uniformly bounded geometry} if $\sup_n N_{X_n}(R) < \infty$ for each $R\geq 0$. Their \emph{coarse disjoint union} is a metric space $(X,d)$ where $X = \bigsqcup_{n\in \NN} X_n$ as a set, the restriction of $d$ to each $X_n$ is $d_n$ and $d$ satisfies $d(X_n, X_m) \to \infty$ as $n+m \to \infty$ and $n \neq m$. We call a metric space \(X\) \emph{sparse} if $X$ has this form.
%Clearly, $\{(X_n,d_n)\}_{n\in \NN}$ has uniformly bounded geometry \emph{if and only if} their coarse disjoint union has bounded geometry.
For such $X$ and a Hilbert space $\H_0$, an operator $T\in \B(\ell^2(X;\H_0))$ is called \emph{block-diagonal} if $T=\SOT$-$\sum_{n\in \NN} T_n$, where each $T_n \in \B(\ell^2(X_n;\H_0))$.
%Such $T$ is said to have \emph{block-rank at most $M$} if each $T_n$ has rank at most $M$, and \emph{block-rank-$M$} if each $T_n$ has rank exactly $M$.

\begin{defn}\label{defn:ppg and quasi-local}
Let $(X,d)$ be a discrete metric space, $\H_0$ be a Hilbert space and $T \in \B(\ell^2(X;\H_0))$.
\begin{enumerate}
  \item The \emph{propagation} of $T$ is defined to be $\ppg(T):=\sup\{d(x,y): T_{xy} \neq 0\}$.
  We say that $T$ has \emph{finite propagation} if $\ppg(T) < \infty$.
  \item We say that $T$ is \emph{quasi-local} if for any $\varepsilon>0$, there exists $R>0$ such that for any $A,B \subseteq X$ with $d(A,B) > R$, we have $\|\chi_A T \chi_B\| < \varepsilon$.
  %Given $\varepsilon, R>0$, say that $T$ is \emph{$(\varepsilon,R)$-quasi-local} if for any $A,B \subseteq X$ with $d(A,B) > R$, we have $\|\chi_A T \chi_B\| < \varepsilon$.
  \item We say that $T$ is \emph{locally compact} if \(\chi_FT\) and \(T\chi_F\) are compact for every finite subset \(F\subseteq X\).
%  each $T_{x,y}$ is a compact operator on $\H_0$.
\end{enumerate}
\end{defn}

\begin{defn}\label{defn:Roe and quasi-local algebra}
Let $(X,d)$ be a discrete metric space of bounded geometry and $\H_0$ a Hilbert space.
\begin{enumerate}
  \item The \emph{Roe algebra} of $X$, denoted by $C^*(X;\H_0)$, is the norm closure of all finite propagation and locally compact operators in $\B(\ell^2(X;\H_0))$.
  %The \emph{algebraic Roe algebra} of $X$ (with coefficient $\H_0$), denoted by $\CC[X;\H_0]$, is the set of all finite propagation and locally compact operators in $\B(\ell^2(X;\H_0))$. The \emph{Roe algebra} of $X$ (with coefficient $\H_0$), denoted by $C^*(X;\H_0)$, is the norm closure of $\CC[X;\H_0]$.
  \item The \emph{quasi-local algebra} of $X$, denoted by $C^*_q(X;\H_0)$, is the set of all quasi-local and locally compact operators in $\B(\ell^2(X;\H_0))$.
\end{enumerate}
\end{defn}

When $\H_0 = \CC$, simply denote $C^*_u(X) = C^*(X;\CC)$, called the \emph{uniform Roe algebra} of $X$. Also $C^*_{uq}(X)=C^*_q(X;\CC)$ is the \emph{uniform quasi-local algebra}.
When $\H_0$ is infinite-dimensional and separable, $C^*(X;\H_0)$ is the usual \emph{Roe algebra of $X$}, and $C^*_q(X;\H_0)$ is the usual \emph{quasi-local algebra of $X$}.

We recall the notion of ghost operators:

\begin{defn}[G. Yu]\label{defn:ghost}
Let $(X,d)$ be a discrete metric space of bounded geometry and $\H_0$ a Hilbert space. Say that $T \in \B(\ell^2(X;\H_0))$ is a \emph{ghost} operator if for any $\varepsilon>0$, there exists a finite subset $F \subseteq X$ such that $\|T_{xy}\| < \varepsilon$ for any $(x,y) \in (X \times X) \setminus (F \times F)$.
\end{defn}

Note that all ghost operators in the Roe algebra form a closed $\ast$-ideal, denoted by $G(X;\H_0)$. Similarly, all ghost operators in the quasi-local algebra form a norm-closed $\ast$-ideal, denoted by $G_q(X;\H_0)$. When $\H_0 = \CC$, we simply write $G_u(X)$ and $G_{uq}(X)$ for $G(X;\CC)$ and $G_q(X;\CC)$, respectively.

Finally, we recall the notion of Schur multipliers. For a set $X$, we say that a kernel $\varphi:X\times X\to\CC$ is a \emph{Schur multiplier kernel} if
\begin{equation}\label{EQ:Schur multipler}
 M_\varphi(A)=\left(\varphi(x,y)A_{xy}\right)_{x,y\in X} \quad \text{for} \quad A=(A_{xy})_{x,y\in X} \in \B(\ell^2(X))
\end{equation}
is a bounded map on $\B(\ell^2(X))$, whose operator norm is denoted by $\|\varphi\|_S$. We also call $M_\varphi$ the associated \emph{Schur multiplier} of $\varphi$. Note that a normalised positive-definite kernel defines a unital completely positive Schur multiplier. Also note that every bounded Schur multiplier preserves compact operators and finite propagation and hence, it preserves $C_u^*(X)$. It also preserves $C_{uq}^*(X)$ since $\chi_AM_\varphi(T)\chi_B=M_\varphi(\chi_AT\chi_B)$ and therefore, $\|\chi_AM_\varphi(T)\chi_B\| \leq \|\varphi\|_S \cdot \|\chi_AT\chi_B\|$.

\section{Generalised expanders and matrix obstructions}\label{sec:obstructions}

This section collects the matrix obstructions to coarse embeddability, derived from generalised expanders. This is one of the key ingredients in the proof for our characterisations.

%used in the first three characterisations coming from generalised expanders. We first recall the notion of generalised expanders and their sparse refinements, and then compare Poincar\'e forms and obtain a pair of matrices whose spectral and propagation properties will be used independently in Sections~\ref{sec:cuts} and~\ref{sec:onsl}.

\subsection{Generalised expanders and sparse refinements}
We recall the following:

\begin{defn}[\cite{Tes09}]\label{defn:generalised expanders}
Given $k,r > 0$, a finite metric space $(X,d)$ is called a \emph{generalised $(k,r)$-expander} if there is a symmetric probability measure $\mu$ on $X\times X$ such that $\supp(\mu) \subseteq \{(x,y) \in X^2: d(x,y)\geq r\}$ and
\begin{equation}\label{EQ:generalised expander}
 \sum_{x,y \in X}\|F(x)-F(y)\|^2 \mu(x,y) \leq k^2
\end{equation}
for any Hilbert-space-valued map $F:X \to \H$ satisfying
\[
\|F(x) - F(y)\|_{\H} \leq d(x,y) \quad \text{whenever} \quad d(x,y) \geq 1.
\]
%for any $1$-Lipschitz map $F:X \to \H$ for some Hilbert space $\H$.

A sequence of finite metric spaces $\{(X_n,d_n)\}_n$ is called a sequence of \emph{generalised expanders} if there exist $k>0$ and a sequence $r_n \to \infty$ such that each $X_n$ is a generalised $(k,r_n)$-expander.
\end{defn}

The following characterisation is a reformulation of \cite[Theorems 4 and 14]{Tes09}. Indeed, necessity comes from the Hilbert-space case of \cite[Theorem 14]{Tes09} with \(p=2\) and \(\rho_+(t)=t\), and sufficiency is directly from \cite[Theorem 4]{Tes09}.

\begin{prop}\label{prop:tessera-general}
A metric space $X$ does not coarsely embed into Hilbert space \emph{if and only if} there exists a sequence of finite subsets $\{Y_n\}_n$ of $X$ which forms a sequence of generalised expanders.
\end{prop}

For locally finite metric spaces, the sequence $\{Y_n\}_n$ can be made sparse as follows. This reduction is useful to prove our main result later.

\begin{prop}\label{prop:sparse-reduction}
Let $(X,d)$ be a locally finite metric space. Then $X$ does not coarsely embed into Hilbert space \emph{if and only if} there exists a sequence of disjoint finite subsets $\{Y_n\}_n$ of $X$ with $d(Y_n, Y_m) \to \infty$ as $n+m \to \infty$ and $n \neq m$ such that $\{Y_n\}_n$ is a sequence of generalised expanders.
\end{prop}

\begin{proof}
We only need to prove necessity. Assume that $X$ does not coarsely embed into Hilbert space.
By Proposition~\ref{prop:tessera-general}, there exist $k>0$, $r_j\to\infty$, finite subsets $Z_j\subseteq X$ and symmetric probability measures $\mu_j$ on $Z_j\times Z_j$ with support contained in $\{(x,y)\in Z_j^2:d(x,y)\geq r_j\}$ such that
\begin{equation}\label{eq:sparse-source-energy}
 \sum_{x,y\in Z_j}\|G(x)-G(y)\|^2\mu_j(x,y)\le k^2,
\end{equation}
for any Hilbert-space-valued map $G:Z_j\to \H$ satisfying $\|G(x) - G(y)\| \leq d(x,y)$ whenever $d(x,y) \geq 1$.
Extend each $\mu_j$ by zero to $X\times X$, still denoted by $\mu_j$.

First, claim that $\mu_j(A\times X)\to 0$ for every finite subset $A\subseteq X$.
Indeed, given $\varepsilon>0$, choose $R>0$ so that $k^2/R^2<\varepsilon$. Consider the $1$-Lipschitz function
\[
 f(z):=\min\{d(z,A),R\} \quad \text{for} \quad z\in X.
\]
Since $f\equiv 0$ on $A$ and $f \equiv R$ on $X\setminus \N_R(A)$, \eqref{eq:sparse-source-energy} gives
\[
 \mu_j\bigl(A\times(X\setminus \N_R(A))\bigr)\le \frac{k^2}{R^2}<\varepsilon.
\]
On the other hand, for any $(x,y) \in A\times \N_R(A)$ we have $d(x,y)\leq \diam(A)+R$. Hence, for sufficiently large $j$, the support condition of $\mu_j$ gives $\mu_j(A\times \N_R(A))=0$. Combining these estimates, we prove the claim.

We now construct $Y_n$ recursively. Choose $j_1$ with $r_{j_1} \geq 1$, and set $Y_1 = Z_{j_1}$ and $\nu_1:=\mu_{j_1}$. Suppose that $Y_1,\ldots,Y_{n-1}$ have been chosen for some $n \geq 2$, and we set $B_n:=\N_n(Y_1\cup\cdots\cup Y_{n-1})$, which is finite. Since $X$ is infinite, choose $o_n\in X\setminus B_n$ and define
\[
 L_n:=1+ \max_{b\in B_n} d(o_n,b)>0 \quad \text{and} \quad A_n:=\N_{L_n}(B_n).
\]
By the claim above, we choose $j_n \geq n$ sufficiently large that $\mu_{j_n}(A_n\times X)<1/4$. Hence, the symmetry of $\mu_{j_n}$ gives
\[
 q_n:=\mu_{j_n}\bigl((X\setminus A_n)^2\bigr)
 \ge 1-2\mu_{j_n}(A_n\times X)>\frac12.
\]
Now set
\begin{equation}\label{eq:sparse-buffer-choice}
 Y_n:=(Z_{j_n}\setminus B_n)\cup\{o_n\}
 \quad \text{and} \quad
 \nu_n:=q_n^{-1}\mu_{j_n}|_{(X\setminus A_n)^2}.
\end{equation}
Then $Y_n$ is finite, and we regard $\nu_n$ as a symmetric probability measure on $Y_n\times Y_n$ with support contained in $\{(x,y)\in Y_n^2:d(x,y)\ge r_{j_n}\}$.

Now we verify the Poincar\'e inequality. Let $F:Y_n\to \H$ be a Hilbert-space-valued map satisfying $\|F(x) - F(y)\| \leq d(x,y)$ whenever $d(x,y) \geq 1$. After translation, we assume $F(o_n)=0$. Set $t_x:=d(x,B_n)$ and $\theta_x:=\min\left\{1,\frac{t_x}{L_n}\right\}$ for $x\in Y_n$, and define a map $G:Z_{j_n}\to \H$ by
\[
 G(x):=
 \begin{cases}
  0,& \text{if } x\in Z_{j_n}\cap B_n, \text{ or }x\in Z_{j_n}\setminus B_n \text{ with } d(x,o_n) < 1;\\
  \theta_x \cdot F(x),& \text{if }x\in Z_{j_n}\setminus B_n \text{ with } d(x,o_n) \geq 1.
 \end{cases}
\]
We claim that $\|G(x) - G(y)\| \leq 3d(x,y)$ whenever $d(x,y) \geq 1$. To see this, fix $x,y\in Z_{j_n}$ with $d(x,y) \geq 1$. Note that when $x\in Z_{j_n}\setminus B_n$ and \(d(x,o_n)\ge1\), we have
\begin{equation}\label{eq:sparse-anchor-bound}
\|F(x)\| =\|F(x)-F(o_n)\| \le d(x,o_n) \le t_x+L_n.
\end{equation}
\textbf{Case I.} $x,y\in Z_{j_n}\setminus B_n$. If both \(G(x)\) and \(G(y)\) are zero, the estimate is trivial. Assume that \(G(x)\neq0\) and \(G(y)\neq0\). Without loss of generality, assume $t_x\le t_y$, so $\theta_x\le\theta_y$. Then
\[
\|G(x)-G(y)\| \leq \theta_y\|F(x)-F(y)\| +(\theta_y-\theta_x)\|F(x)\|.
\]
If $\theta_x=\theta_y$, the right-hand side is at most $d(x,y)$. If $\theta_x<\theta_y$, then $t_x<L_n$ and hence, \eqref{eq:sparse-anchor-bound} shows $\|F(x)\|\le2L_n$ and
$|\theta_x-\theta_y|\le d(x,y)/L_n$. Consequently,
\[
 \|G(x)-G(y)\|
 \le d(x,y)+\frac{d(x,y)}{L_n}\cdot 2L_n =3d(x,y).
\]
If exactly one of \(G(x)\) and \(G(y)\) is zero, by symmetry assume that \(G(x)=0\) and \(G(y)\neq0\). If $F(x)=0$, then $\|G(x)-G(y)\| \leq \theta_y\|F(x)-F(y)\|\leq \|F(x)-F(y)\|$ and the claim holds. Now we only need to consider the case when $F(x)\neq 0$. Since \(x\notin B_n\) and $F(x)\neq 0$, the definition of \(G\) implies $d(x,o_n)<1$. Since \(G(y)\neq0\), we have \(d(y,o_n)\ge1\) and
\[
\|G(y)\|
=\theta_y\|F(y)\|
\le \|F(y)-F(o_n)\|
\le d(y,o_n).
\]
Therefore, we have
\[
\|G(x)-G(y)\|
=\|G(y)\|
\le d(y,x)+d(x,o_n)
< d(x,y)+1
\le2d(x,y).
\]
\textbf{Case II. } $x\in Z_{j_n}\setminus B_n$ and $y\in Z_{j_n}\cap B_n$. We can assume $G(x) \neq 0$, otherwise the result holds trivially. Now \eqref{eq:sparse-anchor-bound} gives
\[
 \|G(x)-G(y)\|=\|G(x)\|
 \le\min\{1,t_x/L_n\} \cdot (t_x+L_n)
 \le2t_x\le2d(x,y).
\]
\textbf{Case III.} $x,y\in B_n$. Then $G(x) = G(y) = 0$. Hence, we prove the claim.

Now applying \eqref{eq:sparse-source-energy} to $G/3$ yields
\[
 \sum_{x,y\in Z_{j_n}}\|G(x)-G(y)\|^2\mu_{j_n}(x,y)\le9k^2.
\]
For \(x\in Z_{j_n}\setminus A_n\), we have \(t_x>L_n\) and hence, \(\theta_x=1\). Moreover, \(d(x,o_n)\ge1\) by the choice of $L_n$. Therefore, $G(x)=F(x)$ for any $x\in Z_{j_n}\setminus A_n$.
Since $\nu_n$ is supported on $(Z_{j_n}\setminus A_n)^2$, it follows that
\[
 \sum_{x,y\in Y_n}\|F(x)-F(y)\|^2\nu_n(x,y) =q_n^{-1}\! \sum_{x,y\in Z_{j_n}\setminus A_n} \|G(x)-G(y)\|^2\mu_{j_n}(x,y)\leq 9q_n^{-1}k^2<18k^2.
\]
Thus, $Y_n$ is a generalised $(3\sqrt2\,k,r_{j_n})$-expander.
Since $d(Y_n, Y_1 \cup \cdots \cup Y_{n-1}) \geq n$, we conclude the proof.
\end{proof}

For a space which is already presented as a coarse disjoint union, one can say more: the obstruction can be chosen along a subsequence of the prescribed components. We record this stronger refinement for later use.

\begin{prop}\label{prop:Tessera_generalised_expander}
Let $\{(X_n,d_n)\}_{n\in \NN}$ be a sequence of finite metric spaces and $(X,d)$ be their coarse disjoint union. Then $X$ can be coarsely embedded into Hilbert space \emph{if and only if} no subsequence $\{X_{n_k}\}_k$ forms generalised expanders.
\end{prop}

%Note that the when $X$ cannot be coarsely embedded into Hilbert space, \cite[Theorem 4]{Tes09} shows the existence of a sequence of generalised expanders, coarsely embedded into $X$. Proposition~\ref{prop:Tessera generalised expander} is stronger in this direction, showing that these generalised expanders can be realised as a subsequence of the blocks.

\begin{proof}
If a subsequence \(\{X_{n_k}\}_k\) forms a sequence of generalised expanders, then \(X\) cannot coarsely embed into Hilbert space by Proposition~\ref{prop:tessera-general}.

Conversely, assume that \(X\) cannot be coarsely embedded into Hilbert space.
For each $n$, choose $a_n\geq n+2\diam(X_n)$. Replace the metric $d$ such that $d(x,y) = a_n + a_m$ for $x\in X_n$ and $y\in X_m$ with $n\neq m$, which retains the same coarse equivalence class. By Proposition~\ref{prop:tessera-general}, there are finite subsets \(Z_j\subseteq X\), \(r_j\to\infty\), \(K_0>0\) and symmetric probability measures \(\widetilde\mu_j\) on \(Z_j^2\) such that $\supp(\widetilde\mu_j)
 \subseteq\{(x,y):d(x,y)\ge r_j\}$ and
\begin{equation}\label{eq:tessera-global}
 \sum_{x,y} \|F(x)-F(y)\|^2\widetilde\mu_j(x,y) \leq K_0^2,
\end{equation}
for any Hilbert-space-valued map $F: Z_j \to \H$ satisfying $\|F(x) - F(y)\| \leq d(x,y)$ whenever $d(x,y) \geq 1$. After discarding finitely many \(j\), we may assume that $r_j>2$.

Let $\{e_n\}_{n}$ be the standard orthonormal basis of $\ell^2(\NN)$ and define $G: X \longrightarrow \ell^2(\NN)$ by $G(x):=a_n e_n$ for $x\in X_n$, which is $1$-Lipschitz by the choice of $d$. If $\widetilde\mu_j(x,y) \neq 0$ for some $x\in X_n$, $y\in X_m$ with $n\ne m$, then
\[
 \|G(x)-G(y)\|^2=a_n^2+a_m^2
 \geq\frac{(a_n+a_m)^2}{2} = \frac{d(x,y)^2}{2} \ge\frac{r_j^2}{2}.
\]
Applying \eqref{eq:tessera-global} to the map $G$ gives
\[
 \widetilde\mu_j\left(\bigcup_{n\ne m}X_n\times X_m\right)
 \le\frac{2K_0^2}{r_j^2}\longrightarrow 0 \quad \text{as} \quad j \to +\infty.
\]
After discarding finitely many $j$, we may therefore assume that
\begin{equation}\label{eq:withinmass}
 \sum_n\widetilde\mu_j(X_n\times X_n)\ge\frac12.
\end{equation}

Fix such a $j$, and write $p_n:=\widetilde\mu_j(X_n\times X_n)$. Since $\widetilde\mu_j$ has finite support, there are only finitely many $n$ such that $p_n>0$. For such $n$, set $\nu_n:=p_n^{-1} \cdot \widetilde\mu_j|_{X_n\times X_n}$ and we have $D_n:=\diam(X_n)\ge r_j>2$.
Define
\[
 P_n:=\sup_F\sum_{x,y\in X_n}
 \|F(x)-F(y)\|^2 \nu_n(x,y),
\]
where the supremum is taken over all Hilbert-space-valued maps $F:X_n\to H$ satisfying $\|F(x) - F(y)\| \leq d(x,y)$ whenever $d(x,y) \geq 1$. Since \(\nu_n\) is supported on pairs \((x,y)\) with \(d(x,y)\ge r_j>1\), we have $P_n\le \diam(X_n)^2<\infty$.

We claim:
\begin{equation}\label{EQ:pnKn}
\sum_n p_nP_n\le K_0^2.
\end{equation}
To see this, for each \(n\) with \(p_n>0\), fix a Hilbert-space-valued map $F_n:X_n\to \H_n$ satisfying $\|F_n(x) - F_n(y)\| \leq d(x,y)$ whenever $d(x,y) \geq 1$. Choose \(u_n,v_n\in X_n\) with $d(u_n,v_n)=D_n$. After translating, assume \(F_n(u_n)=0\). Given $x\in X_n$, if \(d(x,u_n)\ge1\), then
$\|F_n(x)\|\le d(x,u_n)\le D_n$. If \(d(x,u_n)<1\), then $d(x,v_n)>D_n-1>1$ and hence,
\[
\|F_n(x)\| \le \|F_n(x)-F_n(v_n)\|+\|F_n(v_n)\|\le d(x,v_n)+d(u_n,v_n)\le2D_n \leq a_n.
\]
Consider $F:X \to \bigoplus_n \H_n$ such that $F|_{X_n}:=F_n$ for those finitely many $n$ with $p_n>0$, and $0$ otherwise. By the choice of $d$, for \(x\in X_n\) and \(y\in X_m\) with \(n\neq m\), we have
\[
\|F(x)-F(y)\|
\le \sqrt{a_n^2+a_m^2}
\le a_n+a_m
=d(x,y).
\]
Therefore, $\|F(x) - F(y)\| \leq d(x,y)$ whenever $d(x,y) \geq 1$.
Applying \eqref{eq:tessera-global} to $F$, we deduce the following
\[
 K_0^2\ge
 \sum_n\sum_{x,y\in X_n}
 \|F_n(x)-F_n(y)\|^2 \widetilde\mu_j(x,y) = \sum_n p_n \sum_{x,y\in X_n}
 \|F_n(x)-F_n(y)\|^2 \nu_n(x,y).
\]
Since only finitely many \(p_n\) are non-zero, taking the supremum independently over each \(F_n\) proves the claim.

Combining \eqref{eq:withinmass} and \eqref{EQ:pnKn}, there exists $n=n(j) \in \NN$ with $p_n>0$ such that $P_n\leq 2K_0^2$. Set $\mu_j:=\nu_{n(j)}$, whose support is contained in $\big\{(x,y)\in X_{n(j)}^2:d(x,y)\ge r_j\big\}$.
Moreover, for any Hilbert space-valued map $F:X_{n(j)}\to \H$ with $\|F(x) - F(y)\| \leq d(x,y)$ whenever $d(x,y) \geq 1$, we have:
\[
 \sum_{x,y\in X_{n(j)}}
 \|F(x)-F(y)\|^2 \mu_j(x,y) \leq 2K_0^2
\]
Since $r_j \to \infty$ and $\mu_j(X_{n(j)}^2) = 1$, we obtain $n(j) \to +\infty$. Passing to a subsequence, we may assume that \(n(j)\) is strictly increasing. This finishes the proof.
\end{proof}

\subsection{Poincar\'e comparison and matrix obstructions}
The following comparison turns the Hilbert-space-valued Poincar\'e inequality into one between scalar quadratic forms, while controlling the distance cost of the dominating form.

\begin{lem}\label{lem:comparison}
Let $Y$ be a finite metric space with $|Y| \geq 2$, and $\mu$ be a symmetric probability measure on $Y\times Y$ with $\mu(x,x)=0$ for every $x\in Y$. Suppose that $C>0$ satisfies
\begin{equation}\label{eq:comparison-hyp}
 \sum_{x,y\in Y}\|F(x)-F(y)\|^2 \mu(x,y) \leq C
\end{equation}
for any $1$-Lipschitz Hilbert-space-valued map $F: Y \to \H$. Then there is a symmetric non-negative function $\eta$ on $Y \times Y$ with $\eta(x,x)=0$ for every $x\in Y$, such that
\begin{equation}\label{eq:second-moment}
 \sum_{x,y\in Y}d(x,y)^2\eta(x,y)\le C
\end{equation}
and for every $f:Y\to\CC$, we have
\begin{equation}\label{eq:dirichlet-domination}
 \sum_{x,y\in Y}|f(x)-f(y)|^2 \mu(x,y)
 \le
 \sum_{x,y\in Y}|f(x)-f(y)|^2 \eta(x,y).
\end{equation}
\end{lem}

\begin{proof}
Fix $o\in Y$. Let $\mathcal C$ be the set of kernels $\psi(x,y)=\|F(x)-F(y)\|^2$ arising from $1$-Lipschitz Hilbert-space-valued maps $F$ on $Y$. Write $Y=\{o,x_1,\ldots,x_m\}$ and for each kernel $\psi$, define
\[
 G_\psi(i,j):=\frac12\bigl(\psi(x_i,o)+\psi(x_j,o)-\psi(x_i,x_j)\bigr) \quad \text{for} \quad i,j=1,\cdots,m.
\]
Claim: Given a symmetric kernel $\psi$ on $Y$ with $\psi(x,x)=0$ and $0\le\psi(x,y)\le d(x,y)^2$ for any $x,y\in Y$, we have $\psi \in \mathcal C$ if and only if $G_\psi \geq 0$.

Indeed, first assume that $\psi \in \mathcal C$. Replacing $F$ by $F-F(o)$, we assume $F(o)=0$. After realifying the target Hilbert space if necessary, we compute directly that $G_\psi(i,j) = \langle F(x_i),F(x_j)\rangle$ and hence, $G_\psi\ge0$. Conversely, suppose that $G_\psi\ge0$. Then there exist vectors $v_1,\cdots,v_m$ in some real Hilbert space such that $\langle v_i,v_j\rangle=G_\psi(i,j)$. Define $F(o)=0$ and $F(x_i):=v_i$ for $i=1,\cdots, m$. A direct calculation shows that
\[
\|F(x_i)-F(x_j)\|^2 = G_\psi(i,i)+G_\psi(j,j)-2G_\psi(i,j)=\psi(x_i,x_j)
\]
and $\|F(x_i)-F(o)\|^2=\psi(x_i,o)$. Hence, we obtain $\psi \in \mathcal C$. Using this characterisation of $\mathcal C$, it is clear that $\mathcal C$ is compact and convex.

Let $\mathcal P$ be the set of symmetric probability measures on $Y^2$ which vanish on the diagonal. Then $\mathcal P$ is a  compact and convex subset of $\mathbb{R}^{Y \times Y}$. Moreover, we have
\[
 \max_{p\in\mathcal P}\sum_{x\ne y}p(x,y)\frac{\psi(x,y)}{d(x,y)^2}
 =\max_{x\ne y}\frac{\psi(x,y)}{d(x,y)^2}.
\]
To see this, first note that the left-hand side is at most the right-hand side. Conversely, if the
right-hand side is attained at $(x_0,y_0)$, then the measure $p$ defined by $p(x_0,y_0)=p(y_0,x_0)=\frac12$ and $p(x,y)=0$ otherwise, attains the same value.

Scaling \eqref{eq:comparison-hyp} and using the preceding identity give
\begin{equation}\label{eq:homogenized}
 \sum_{x,y}\psi(x,y) \mu(x,y)
 \le C\max_{x\ne y}\frac{\psi(x,y)}{d(x,y)^2} = C\max_{p\in\mathcal P}\sum_{x\ne y}p(x,y)\frac{\psi(x,y)}{d(x,y)^2}.
\end{equation}
For $\psi\in\mathcal C$ and $p\in\mathcal P$, define
\[
 \Lambda(\psi,p) :=C\sum_{x\ne y}p(x,y)\frac{\psi(x,y)}{d(x,y)^2} -\sum_{x,y}\mu(x,y)\psi(x,y).
\]
Then \eqref{eq:homogenized} implies that $\max_{p\in\mathcal P}\Lambda(\psi,p)\ge0$. The function $\Lambda$ is continuous and affine in each variable. Since $\mathcal C$ and $\mathcal P$ are compact and convex, the finite-dimensional
minimax theorem gives
\[
 \max_{p\in\mathcal P}\min_{\psi\in\mathcal C}\Lambda(\psi,p) =\min_{\psi\in\mathcal C}\max_{p\in\mathcal P}\Lambda(\psi,p)\ge0.
\]
Since the function $p\mapsto\min_{\psi\in\mathcal C}\Lambda(\psi,p)$ is continuous on $\mathcal P$, there exists $p_0\in\mathcal P$ such that $\min_{\psi\in\mathcal C}\Lambda(\psi,p_0)\geq 0$. Equivalently, $\Lambda(\psi,p_0)\ge0$ for any $\psi\in\mathcal C$.

Finally, we define
\begin{equation}\label{eq:eta-from-p}
 \eta(x,y)=
 \begin{cases}
 \displaystyle\frac{C\,p_0(x,y)}{d(x,y)^2},&x\ne y,\\[2mm]
 0,&x=y.
 \end{cases}
\end{equation}
Then $\eta$ is symmetric and non-negative, and \eqref{eq:second-moment} follows from:
\[
 \sum_{x,y}d(x,y)^2\eta(x,y)
 =C\sum_{x\ne y}p_0(x,y)=C.
\]
It remains to prove \eqref{eq:dirichlet-domination}. For a $1$-Lipschitz map $f:Y\to\CC$, consider the kernel $\psi_f(x,y):=|f(x)-f(y)|^2$, which belongs to $\mathcal C$. Then $\Lambda(\psi_f, p_0) \geq 0$ implies
\[
 \sum_{x,y}\mu(x,y)|f(x)-f(y)|^2 \le C\sum_{x\ne y}p_0(x,y) \frac{|f(x)-f(y)|^2}{d(x,y)^2}=\sum_{x,y}\eta(x,y)|f(x)-f(y)|^2.
\]
For arbitrary $f:Y\to\CC$, both sides of \eqref{eq:dirichlet-domination} are zero if $f$ is constant. Otherwise,
\[
 L_f:=\max_{x\ne y}\frac{|f(x)-f(y)|}{d(x,y)}
\]
is finite and positive, and $f/L_f$ is $1$-Lipschitz. Applying the preceding
inequality to $f/L_f$, we conclude \eqref{eq:dirichlet-domination} for $f$.
\end{proof}

The comparison produces the following matrix obstruction, where the two families $S_n,M_n$ have different roles. The matrices $S_n$ carry a uniformly non-trivial positive spectral part and admit uniform finite-propagation approximations, whereas the matrices $M_n$ are entry-wise non-negative with supports escaping to infinity. The inequality $S_n\le M_n$ links these two features and is crucial below.

\begin{prop}\label{prop:matrix-obstruction}
Let $(X,d)$ be a discrete metric space of bounded geometry which does not coarsely embed into Hilbert space. Then there exist a sparse subspace $Y=\bigsqcup_{n \in \NN} Y_n\subseteq X$, a constant $C\ge1$, a sequence $r_n\to\infty$, and self-adjoint matrices $S_n,M_n\in \B(\ell^2(Y_n))$ satisfying the following conditions:
\begin{enumerate}
\item $\|S_n\|\le2$ with the maximal eigenvalue $\lambda_{\max}(S_n)\ge C^{-1}$;
\item $S_n\le M_n$ and $\|M_n\|\le1$;
\item $(M_n)_{xy}\ge0$ for all $x,y\in Y_n$, and $(M_n)_{xy}=0$ whenever $d(x,y)<r_n$;
\item Setting $S_n^{(R)}$ to be the restriction of $S_n$ on $E_R$ for $R\geq 1$, we have $\|S_n-S_n^{(R)}\|\le R^{-2}$.
\end{enumerate}
In particular, $S=\mathrm{(SOT)}$-$\sum_{n} S_n$ belongs to $C_u^*(X)$.
\end{prop}

We isolate the following estimate needed in the proof:

\begin{lem}\label{lem:weighted Schur test}
For a finite set $Y$, a symmetric matrix $K$ on $Y$ with non-negative entries and weights $w_x>0$ such that $\sum_y K_{xy}w_y\leq a w_x$, we have $\|K\|\le a$.
\end{lem}

\begin{proof}
For any $\xi, \zeta \in \ell^2(Y)$, we have
\[
 |\langle\xi,K\zeta\rangle| \leq \sum_{x,y}K_{xy}|\xi_x|\cdot |\zeta_y| \leq\left(\sum_{x,y}K_{xy}\frac{w_y}{w_x}|\xi_x|^2\right)^{1/2}\left(\sum_{x,y}K_{xy}\frac{w_x}{w_y}|\zeta_y|^2\right)^{1/2}
 \leq a\|\xi\|\cdot\|\zeta\|.
\]
This shows that $\|K\|\le a$.
\end{proof}

\begin{proof}[Proof of Proposition~\ref{prop:matrix-obstruction}]
Replacing \(d\) if necessary, by the coarsely equivalent metric
\[
\widetilde d(x,y)=
\begin{cases}
0,& \text{if }x=y,\\
\max\{d(x,y),1\},& \text{if }x\ne y,
\end{cases}
\]
we may assume throughout this proof that \(d(x,y)\ge1\) for \(x\ne y\). This replacement does not change the coarse structure and the uniform Roe algebra.

By Proposition~\ref{prop:sparse-reduction}, choose a sparse subspace $Z=\bigsqcup_n Z_n\subseteq X$ whose components form a sequence of generalised $(k,r_n)$-expanders for some $k\geq 1$ and $r_n \to +\infty$, witnessed by measures $\mu_n$ on $Z_n \times Z_n$. Since $d(x,y) \geq 1$ for $x\neq y\in X$, we have
\[
\sum_{x,y\in Z_n}\|F(x)-F(y)\|^2 \mu_n(x,y) \leq k^2
\]
for any $1$-Lipschitz Hilbert-space-valued map $F: Z_n \to \H$. After discarding finitely many terms, we may also assume \(r_n>1\) in the sequel, so the conclusions obtained below are unchanged for the original metric.

Applying Lemma~\ref{lem:comparison} to each $Z_n$ with constant $C:=k^2$, we obtain symmetric non-negative functions $\eta_n$ on $Z_n \times Z_n$ satisfying \eqref{eq:second-moment} and \eqref{eq:dirichlet-domination}. For $x\in Z_n$, set
\[
 m_n(x):=\sum_{y\in Z_n}\mu_n(x,y), \quad c_n(x):=\sum_{y \in Z_n}\eta_n(x,y) \quad \text{and} \quad h_n(x):=\sum_{y\in Z_n} d(x,y)^2\eta_n(x,y).
\]
Taking $f=\chi_{\{x\}}$ in \eqref{eq:dirichlet-domination} gives $m_n(x)\leq c_n(x)$. By the choice of the metric, we obtain
\begin{equation}\label{eq:weights}
 0\le m_n\le c_n\le h_n \quad \text{and} \quad 1\leq Q_n:=\sum_xh_n(x)\le C.
\end{equation}

For each $n$, set $Y_n:=\{x \in Z_n: h_n(x)>0\}$. It follows from \eqref{eq:weights} that on $Z_n \setminus Y_n$, $m_n(x) = c_n(x) = 0$, so $\mu_n$ and $\eta_n$ are supported on $Y_n \times Y_n$. Given $f: Y_n \to \CC$, denote by $D_f$ the diagonal operator on $\ell^2(Y_n)$. Define the following $Y_n$-by-$Y_n$ matrices:
\[
 A_{\mu_n}:=\left(\mu_n(x,y)\right)_{x,y\in Y_n}, \quad A_{\eta_n}=\left(\eta_n(x,y)\right)_{x,y\in Y_n}
\]
and
\[
L_{\mu_n}:=D_{m_n} - A_{\mu_n}, \quad L_{\eta_n}:=D_{c_n}-A_{\eta_n}.
\]
Also define:
\begin{equation*}\label{eq:matrix-form}
 M_n=D_{h_n}^{-1/2}A_{\mu_n}D_{h_n}^{-1/2} \quad \text{and} \quad
 S_n=D_{h_n}^{-1/2}(D_{m_n}-L_{\eta_n})D_{h_n}^{-1/2}.
\end{equation*}

For the weight $w_x=\sqrt{h_n(x)}$ on $Y_n$, the matrix $M_n$ and the off-diagonal part $S'_n$ of $S_n$, we have
\[
 \sum_y(M_n)_{xy}w_y
 =\frac{m_n(x)}{h_n(x)}w_x\le w_x
 \quad \text{and} \quad
 \sum_y(S'_n)_{xy}w_y
 =\frac{c_n(x)}{h_n(x)}w_x\le w_x.
\]
Applying Lemma~\ref{lem:weighted Schur test}, we obtain $\|M_n\|\le1$ and $\|S'_n\|\le1$. Since $M_n$ is self-adjoint, we have $M_n \leq \Id_{\ell^2(Y_n)}$. The diagonal part of $S_n$ has norm at most $1$ by \eqref{eq:weights} and hence, $\|S_n\|\le2$. On the other hand, for each $R>0$, set $S_n^{(R)}$ to be the restriction of $S_n$ on the $R$-diagonal of $Y_n\times Y_n$. Since
\[
\sum_y(S_n - S_n^{(R)})_{xy}w_y = \sum_{y:d(x,y)>R}(S'_n)_{xy}w_y
 =\frac1{\sqrt{h_n(x)}}\sum_{y:d(x,y)>R}\eta_n(x,y)
 \le R^{-2}w_x,
\]
the fact above implies that $\|S_n - S_n^{(R)}\| \leq R^{-2}$. Hence, $S:=\mathrm{(SOT)}$-$\sum_n S_n \in C^*_u(X)$.

Since $L_{\mu_n}, L_{\eta_n}$ are symmetric, \eqref{eq:dirichlet-domination} implies $L_{\mu_n} \leq L_{\eta_n}$ and hence, $D_{m_n}-L_{\eta_n}\le A_{\mu_n}$. This implies $S_n \leq M_n$. Define a unit vector $u_n \in \ell^2(Y_n)$ by $u_n(x):=\sqrt{h_n(x)/Q_n}$. Then we have
\[
\langle u_n,S_nu_n\rangle = \frac1{Q_n}\left(\sum_x(m_n(x)-c_n(x))
                     +\sum_{x\ne y}\eta_n(x,y)\right)=\frac1{Q_n}\ge C^{-1}.
\]
This finishes the proof.
\end{proof}

\section{Cut decompositions and Schur approximation}\label{sec:cuts}

%This section is devoted to another key tool coming from coarse embedding. More precisely, we show a cut decomposition of the Hilbert metric, and use the associated cuts to control Schur multiplication on matrix coefficients.

The purpose of this section is to establish a quantitative Schur approximation associated to a coarse embedding. The main geometric input is a finite cut decomposition of the Hilbert metric. This is the second key tool needed in the proof of our characterisations.

%\subsection{Cut decompositions of Hilbert metrics}\label{subsec:cut-tools}

We first fix some notation. A \emph{cut} of a finite set $V$ is a partition $V=C\sqcup C^c$, represented by a subset $C\subseteq V$. Associated to a cut $C$, define the \emph{crossing set} by
\begin{equation}\label{EQ:DC}
 D_C:=(C\times C^c)\cup(C^c\times C)\subseteq V^2,
\end{equation}
consisting of pairs of points which lie on opposite sides of the cut $C$. Hence for $u,v\in V$, we have $\chi_{D_C}(u,v)=|\chi_C(u)-\chi_C(v)|$.
%When a family of cuts $\{C_j\}_j$ is under consideration, we simply write $D_j:=D_{C_j}$.

The following lemma is the basic geometric input, which produces a convex decomposition of the Hilbert metric using cut representation.

\begin{lem}\label{lem:cuts}
For a finite subset $V$ of a Hilbert space $\H$ and $t>0$, there exist finitely many maps $c_1, \cdots, c_N: V \to \{-1,1\}$ and weights $\lambda_j \geq 0$ with $\sum_j\lambda_j=1$ such that
\[
\sum_{j=1}^N\lambda_jc_j(u)c_j(v) =e^{-t\|u-v\|} \quad \text{for any} \quad u,v \in V.
\]
\end{lem}

\begin{proof}
Let us realify $\H$, translate one point in $V$ to the origin, and identify the linear span of $V$ with $\mathbb{R}^d$. The case $d=0$ holds trivially, and we assume $d\geq 1$.
Let $\sigma$ be the probability Lebesgue measure on $S^{d-1}$, $e_1=(1,0,\cdots,0) \in S^{d-1}$ and
\[
 c_d=\left(\int_{S^{d-1}} |\langle e_1, \vartheta\rangle| \d\sigma(\vartheta)\right)^{-1}.
\]
By rotation invariance and homogeneity, for any $w\in\mathbb{R}^d$ we have
\begin{equation}\label{EQ:norm of w}
 c_d\int_{S^{d-1}}|\langle w,\vartheta\rangle|\d\sigma(\vartheta)=\|w\|.
\end{equation}
By the layer-cake identity, for any $u,v \in \H$ and $\vartheta \in S^{d-1}$ we have
\[
|\langle u-v,\vartheta\rangle| = \int_{\mathbb{R}} \left| \chi_{(-\infty,\langle u,\vartheta\rangle)}(\alpha) - \chi_{(-\infty,\langle v,\vartheta\rangle)}(\alpha)\right|\d\alpha.
\]
Combining with \eqref{EQ:norm of w}, we obtain:
\[
 \|u-v\|= c_d \int_{S^{d-1}} \int_{\mathbb{R}} \left| \chi_{(-\infty,\langle u,\vartheta\rangle)}(\alpha) -
 \chi_{(-\infty,\langle v,\vartheta\rangle)}(\alpha) \right| \d\alpha\d\sigma(\vartheta).
\]

For each $\vartheta\in S^{d-1}$ and $\alpha\in\mathbb{R}$, set $E_{\vartheta,\alpha}:=\{z\in V:\langle z,\vartheta\rangle>\alpha\}$.
Then we have $\chi_{(-\infty,\langle z,\vartheta\rangle)}(\alpha) = \chi_{E_{\vartheta,\alpha}}(z)$. Thus each pair $(\vartheta,\alpha)$ determines a cut $V=E_{\vartheta,\alpha}\sqcup (V \setminus E_{\vartheta,\alpha})$.
Since $V$ is finite, only finitely many distinct cuts occur. Grouping together all $(\vartheta,\alpha)$ which determine the same cut, we obtain subsets $E_1,\cdots,E_m\subseteq V$ and numbers $a_1,\cdots,a_m\ge0$ such that
\begin{equation}\label{eq:l1cuts}
 \|u-v\| = \sum_{\ell=1}^m a_\ell \left| \chi_{E_\ell}(u)-\chi_{E_\ell}(v) \right| = \sum_{\ell=1}^m a_\ell \chi_{D_{E_\ell}}(u,v).
\end{equation}

For $\ell=1,\cdots, m$, set $p_\ell:=(1-e^{-t a_\ell})/2$. For $\eta=(\eta_1,\cdots,\eta_m)\in\{0,1\}^m$, define
\[
 c_\eta(u) := (-1)^{\sum_{\ell=1}^m \eta_\ell\chi_{E_\ell}(u)} = \prod_{\ell=1}^m (-1)^{\eta_\ell\chi_{E_\ell}(u)} \quad \text{and} \quad
 \lambda_\eta:=\prod_{\ell=1}^m p_\ell^{\eta_\ell}(1-p_\ell)^{1-\eta_\ell}.
\]
Clearly each $\lambda_\eta \geq 0$, and
\[
 \sum_{\eta\in\{0,1\}^m}\lambda_\eta =
 \prod_{\ell=1}^m \left( \sum_{\eta_\ell=0}^1 p_\ell^{\eta_\ell}(1-p_\ell)^{1-\eta_\ell} \right) =
 \prod_{\ell=1}^m\bigl(p_\ell+(1-p_\ell)\bigr) =1.
\]
For each $\ell =1,\cdots,m$ and $u,v\in V$, we have
\begin{equation}\label{EQ:p ell}
 (1-p_\ell) + p_\ell (-1)^{\left|\chi_{E_\ell}(u)-\chi_{E_\ell}(v)\right|} =
 \exp\left(-ta_\ell \left|\chi_{E_\ell}(u)-\chi_{E_\ell}(v)\right|\right).
\end{equation}
Indeed, if $u$ and $v$ lie on the same side of the cut $E_\ell$, both sides equal $1$. If they lie on opposite sides, the left-hand side is $1-2p_\ell=e^{-ta_\ell}$. Therefore, we have
\begin{align*}\label{EQ:cut estimate compute}
\sum_{\eta\in\{0,1\}^m}
\lambda_\eta c_\eta(u)c_\eta(v)
&=
\prod_{\ell=1}^m \left(\sum_{\eta_\ell=0}^1p_\ell^{\eta_\ell}(1-p_\ell)^{1-\eta_\ell} (-1)^{\eta_\ell(\chi_{E_\ell}(u)+\chi_{E_\ell}(v))} \right) \notag\\
&=
 \prod_{\ell=1}^m\left( (1-p_\ell) + p_\ell(-1)^{\left|\chi_{E_\ell}(u)-\chi_{E_\ell}(v)\right|}\right) \notag\\
&=
 \exp\left(-t \sum_{\ell=1}^m a_\ell |\chi_{E_\ell}(u)-\chi_{E_\ell}(v)|\right) \notag\\
&=
 e^{-t\|u-v\|},
\end{align*}
where the penultimate equality comes from \eqref{EQ:p ell} and the last from \eqref{eq:l1cuts}.
%Finally, define the cut $C_\eta:=\{u\in V:c_\eta(u)=1\}$ and denote its crossing set by $D_\eta:=D_{C_\eta}$. Then $c_\eta(u) c_\eta(v)=1-2\chi_{D_\eta}(u,v)$ and hence, we conclude the proof by the computation above.
\end{proof}

Given $t>0$, define the function $q_t: [0,\infty) \to [0,\infty)$ as follows:
\begin{equation}\label{EQ:function qt}
 q_t(s):=\frac{1-e^{-ts}}{2}\quad \text{for} \quad s \geq 0.
\end{equation}
Translating Lemma~\ref{lem:cuts} into the language of cuts, we obtain the following:

\begin{cor}\label{cor:cuts}
For a finite subset $V$ of a Hilbert space $\H$ and $t>0$, there exist finitely many cuts $C_1,\cdots,C_N$ of $V$ and weights $\lambda_j \geq 0$ with $\sum_j\lambda_j=1$ such that
\[
\sum_{j=1}^N\lambda_j\chi_{D_{C_j}}(u,v)=q_t(\|u-v\|) \quad \text{for any} \quad u,v \in V.
\]
\end{cor}

\begin{proof}
Let $c_1, \cdots, c_N: V \to \{-1,1\}$ be the maps provided by Lemma~\ref{lem:cuts}. For each map $c_j$, define a cut $C_j:=c_j^{-1}(\{1\})$. Then for each $u,v\in V$, we have $c_j(u)c_j(v) = 1-2\chi_{D_{C_j}}(u,v)$. This finishes the proof.
\end{proof}

We now pass from the cut decomposition to operators. The following matrix-coefficient Schur estimate is the key tool in the sequel.

\begin{prop}\label{lem:ql-approx}
Let $X$ be a discrete metric space of bounded geometry, and $F:X\to \H$ be a coarse map with control function $\rho_+$. For $t>0$, put
\begin{equation}\label{EQ:phi t}
 \varphi_t(x,y):=e^{-t\|F(x)-F(y)\|} \quad \text{for} \quad x,y\in X.
\end{equation}
Then for any $T\in\B(\ell^2(X))$, $\xi,\eta\in\ell^2(X)$ and $R>0$, we have
\begin{equation}\label{eq:matrix-coefficient-schur}
\big|\langle \eta,(T-M_{\varphi_t}(T))\xi\rangle\big| \le 2\sup_{d(A,B)>R}
\big|\langle \eta,(\chi_AT\chi_B+\chi_BT\chi_A)\xi\rangle\big| +4\|T\|\cdot\|\xi\|\cdot\|\eta\|\sqrt{p_t(R)},
\end{equation}
where $p_t(R):=N_X(R) \cdot q_t(\rho_+(R))$.
\end{prop}

\begin{proof}
We first prove \eqref{eq:matrix-coefficient-schur} when $\xi$ and $\eta$ have finite support. Choose a finite subset $Y\subseteq X$ containing both supports and set $T_Y:=\chi_YT\chi_Y$, viewed as an operator on $\ell^2(Y)$.
Applying Lemma~\ref{lem:cuts} to $F(Y)$ and pulling the resulting maps back to $Y$, we obtain maps $c_j:Y\to\{-1,1\}$ and weights $\lambda_j\ge0$ with $\sum_j\lambda_j=1$ such that
\begin{equation}\label{eq:color-representation-schur}
 \varphi_t(x,y)=\sum_j\lambda_jc_j(x)c_j(y)
 \quad \text{for any} \quad x,y\in Y.
\end{equation}
For each $j$, define the cut $C_j:=\{v\in Y: c_j(v)=1\}$ and denote $D_j:=D_{C_j}$. Then \eqref{eq:color-representation-schur} can be rewritten
\begin{equation}\label{eq:cut-representation-schur}
 q_t(\|F(x)-F(y)\|) = \sum_j\lambda_j\chi_{D_j}(x,y) \quad \text{for any} \quad x,y\in Y.
\end{equation}

To decompose $M_{\varphi_t}(T)$ accordingly, for each $j$ let $U_j\in\B(\ell^2(Y))$ be the diagonal unitary defined by $U_j\delta_x=c_j(x)\delta_x$ for $x\in Y$, and put $S_j:=T_Y-U_jT_YU_j$. Then
\[
 \chi_Y(T-M_{\varphi_t}(T))\chi_Y=\sum_j\lambda_jS_j, \qquad \|S_j\| \leq 2\|T_Y\| \leq 2\|T\|,
\]
and
\begin{equation}\label{EQ:Sj}
 (S_j)_{xy} =\left(1-c_j(x)c_j(y)\right)T_{xy} =2\chi_{D_j}(x,y) T_{xy}.
\end{equation}

Fix $R>0$. For each $j$, consider the $R$-boundary
\[
 E_j:=\{x\in Y:\text{there exists }z\in B_Y(x,R)\text{ with }(x,z)\in D_j\}.
\]
Then we have $d(C_j^+,C_j^-)>R$ for $C_j^+:=C_j\setminus E_j$ and $C_j^-:=C_j^c\setminus E_j$.
Moreover, for $x\in Y$, \eqref{eq:cut-representation-schur} shows that
\[
 \sum_j\lambda_j\chi_{E_j}(x)
\leq \sum_j\lambda_j \sum_{z\in B_Y(x,R)}\chi_{D_j}(x,z) = \sum_{z\in B_Y(x,R)}q_t(\|F(x)-F(z)\|)\leq p_t(R).
\]
For each $j$, let $P_j:=\chi_{E_j}$ and $Q_j:=\Id_{\ell^2(Y)}-P_j$. Then the inequality above shows
\begin{equation}\label{EQ:estimate of P}
\sum_j \lambda_j P_j\leq p_t(R) \Id_{\ell^2(Y)}.
\end{equation}
Moreover, note that $\ran(Q_j) = \ell^2(C_j^+) \oplus \ell^2(C_j^-)$.  Hence,  with respect to this decomposition, \eqref{EQ:Sj} implies:
\[
 Q_jS_jQ_j
 =2\begin{pmatrix}
 0&\chi_{C_j^+}T\chi_{C_j^-}\\
 \chi_{C_j^-}T\chi_{C_j^+}&0
 \end{pmatrix}.
\]

On the other hand, we have the decomposition:
\[
\sum_j \lambda_j S_j = \sum_j \lambda_j Q_jS_jQ_j + \sum_j \lambda_j P_jS_j + \sum_j \lambda_j Q_jS_jP_j.
\]
For any $\xi, \eta \in \ell^2(Y)$, we have
\begin{align*}
\left|
\left\langle
\eta,
\sum_j\lambda_jQ_jS_jQ_j\xi
\right\rangle
\right|
\le
2\sup_{d(A,B)>R}
\left|
\left\langle
\eta,
(\chi_AT\chi_B+\chi_BT\chi_A)\xi
\right\rangle
\right|.
\end{align*}
Next, by the Cauchy--Schwarz inequality, \eqref{EQ:estimate of P} and $\|S_j\|\le2\|T\|$, we obtain:
\begin{align*}
 \left|\left\langle\eta,
 \sum_j\lambda_jP_jS_j\xi\right\rangle\right|
 &=\left|\sum_j\lambda_j\langle P_j\eta,S_j\xi\rangle\right|\leq
 \left(\sum_j\lambda_j\|P_j\eta\|^2\right)^{1/2}
 \left(\sum_j\lambda_j\|S_j\xi\|^2\right)^{1/2}\\
 &\leq
 \left\langle\eta, \left(\sum_j\lambda_jP_j\right)\eta\right\rangle^{1/2}
 \cdot 2\|T\|\cdot \|\xi\| \leq 2\|T\|\cdot \|\eta\|\cdot \|\xi\| \sqrt{p_t(R)}.
\end{align*}
Similarly,
\[
\left|
\left\langle
\eta,
\sum_j\lambda_jQ_jS_jP_j\xi
\right\rangle
\right|=
\left|
\sum_j\lambda_j
\langle S_j^*Q_j\eta,P_j\xi\rangle
\right|\le
2\|T\|\cdot \|\eta\|\cdot \|\xi\| \sqrt{p_t(R)}.
\]
Combining these estimates proves \eqref{eq:matrix-coefficient-schur} for finitely supported $\xi,\eta$. The general case follows by approximating $\xi$ and $\eta$ by finitely supported ones.
%Hence, $\left\|\sum_j\lambda_jP_jS_j\right\| \leq 2\sqrt p \|T\|$. Similarly, we also have $\left\|\sum_j\lambda_jQ_jS_jP_j\right\| \leq 2 \sqrt p \|T\|$.
%Combining the above together, we obtain:
%\[
%\left\|\chi_Y(T - M_{\varphi_t}(T)) \chi_Y \right\| = \left\|\sum_j\lambda_jS_j\right\| \leq 2\sup_{d(A,B)>R}\|\chi_AT\chi_B\| + 4 \sqrt p \|T\|,
%\]
%which concludes \eqref{eq:ql-approx}.
%
%For the last sentence in the statement, given $\varepsilon>0$, take $R>0$ such that $\sup_{d(A,B)>R}\|\chi_AT\chi_B\| < \varepsilon$. Then take $t>0$ such that $N_X(R) \cdot q_t(\rho_+(R)) < \varepsilon$. Hence, we obtain $\left\|T - M_{\varphi_t}(T)\right\| \leq 2 \varepsilon + 4\|T\|\sqrt{\varepsilon}$. This finishes the proof.
\end{proof}

Consequently, we obtain:

\begin{cor}\label{cor:ql-approx}
Let $X$ be a discrete metric space of bounded geometry, and $F:X\to \H$ be a coarse map with control function $\rho_+$. For $t>0$, define $\varphi_t$ as in \eqref{EQ:phi t}. Then for any $T\in\B(\ell^2(X))$ and $R>0$, we have
\begin{equation}\label{eq:ql-approx}
 \|M_{\varphi_t}(T)-T\| \leq 2\sup_{d(A,B)>R}\|\chi_A T\chi_B\| +4\|T\|\sqrt{N_X(R) \cdot q_t(\rho_+(R))},
\end{equation}
Consequently, $M_{\varphi_t}(T) \to T$ as $t \to 0$ for every $T \in C^*_{uq}(X)$.
\end{cor}

\section{Characterisation I: escaping quasi-local couplings}\label{sec:char I}

This section is devoted to Theorem~\ref{thm:esc and ql coupling} and its applications. Recall from Definition~\ref{defn:escaping quasi-local coupling} that an \emph{escaping quasi-local coupling} on a discrete metric space $X$ of bounded geometry  is a sequence of symmetric finitely supported probability measures $\{\nu_n\}_n$ on $X\times X$ satisfying:
\begin{description}[
  font=\normalfont\bfseries,
  leftmargin=3em,
  labelsep=0.5em,
  style=multiline
]
\item[\textbf{(E)}] For any $S>0$, $\nu_n\bigl(\{(x,y)\in X^2:d(x,y)\le S\}\bigr)\to 0$;
\item[\textbf{(Q)}] For any $\varepsilon>0$, there is $R>0$ such that $\nu_n(A\times B)<\varepsilon$ for any $n$ and $A,B \subseteq X$ with $d(A,B)>R$.
\end{description}

\begin{rem}
Note that if $\supp(\nu_n) \subseteq Y_n \times Y_n$ for some $Y_n\subseteq X$, then in verifying condition \textbf{(Q)} it suffices to consider subsets \(A,B\subseteq Y_n\). Indeed, for arbitrary \(A,B\subseteq X\) we have $\nu_n(A\times B) = \nu_n\bigl((A\cap Y_n)\times(B\cap Y_n)\bigr)$, while $d(A\cap Y_n,B\cap Y_n)\geq d(A,B)$.
\end{rem}

We claim in Theorem~\ref{thm:esc and ql coupling} that $X$ can be coarsely embedded into some Hilbert space if and only if it admits no escaping quasi-local coupling. To prove it, we will show that generalised expanders provide escaping couplings when coarse embeddability fails, while the converse follows from the matrix-coefficient Schur estimate in Proposition~\ref{lem:ql-approx}. We package the proof into the following subsection.

%\begin{thm}\label{thm:esc and ql coupling}
%A discrete metric space $(X,d)$ of bounded geometry can be coarsely embedded into some Hilbert space \emph{if and only if} $X$ admits no escaping quasi-local coupling.
%\end{thm}

%The non-trivial direction of the proof for Theorem~\ref{thm:esc and ql coupling} is based on a cut decomposition of finite Hilbert metrics. Since the same decomposition tool will also be used for the Schur approximation in Section~\ref{sec:schur-ghost}, we isolate it as follows.

\subsection{Proof of Theorem~\ref{thm:esc and ql coupling}}\label{subsec:couplings}

\begin{lem}\label{lem:generalised expander implies coupling}
Let $\{Y_n\}_n$ be finite subsets of $(X,d)$ which form a sequence of generalised expanders. Extending their witnessing measures by zero to $X\times X$ gives an escaping quasi-local coupling on $X$.
\end{lem}

\begin{proof}
Let $\mu_n$ witness that $Y_n$ is a generalised $(k,r_n)$-expander for some $k>0$ and $r_n\to\infty$, and extend $\mu_n$ by zero to $X\times X$. Condition \textbf{(E)} follows from the support condition. To see condition \textbf{(Q)}, fix $\varepsilon>0$ and take $R>k/\sqrt\varepsilon$. Given $A,B\subseteq X$ with $d(A,B)>R$, consider
\[
 f(z):=\min\{d(z,A),R\}\quad \text{for} \quad z\in X.
\]
Then $f$ is $1$-Lipschitz, equals zero on $A$, and equals $R$ on $B$. Hence,
\[
R^2\mu_n(A\times B) = R^2\mu_n\left((A\cap Y_n)\times(B\cap Y_n)\right)
 \leq \sum_{x,y \in Y_n}|f(x)-f(y)|^2 \mu_n(x,y) \leq k^2,
\]
which implies $\mu_n(A\times B) \leq k^2/R^2 <\varepsilon$ and finishes the proof.
\end{proof}

Now we prove Theorem~\ref{thm:esc and ql coupling}.

\begin{proof}[Proof of Theorem~\ref{thm:esc and ql coupling}]
First, assume that \(X\) does not coarsely embed into any Hilbert space. Proposition~\ref{prop:tessera-general} provides finite subsets $Y_n\subseteq X$ which form a sequence of generalised expanders. Then Lemma~\ref{lem:generalised expander implies coupling} provides an escaping quasi-local coupling on $X$, which shows the sufficiency.

Conversely, let $F:X\to\H$ be a coarse embedding with control functions $\rho_-,\rho_+$, and suppose that $\{\nu_n\}_n$ is an escaping quasi-local coupling. For each $n$, let
\[
 m_n(x):=\sum_{y\in X}\nu_n(x,y),
 \quad \text{and} \quad
 Y_n:=\{x\in X:m_n(x)>0\}.
\]
Since $\nu_n$ is finitely supported and symmetric, $Y_n$ is finite and $\nu_n$ is supported on $Y_n\times Y_n$. Define $v_n\in\ell^2(Y_n)$ and a matrix $K_n\in\B(\ell^2(Y_n))$ by
\[
 v_n(x):=\sqrt{m_n(x)}
 \quad \text{and} \quad
 (K_n)_{xy}:=\frac{\nu_n(x,y)}{\sqrt{m_n(x)m_n(y)}}
 \quad \text{for} \quad x,y\in Y_n.
\]
Then $\|v_n\|=1$ and $\sum_{y\in Y_n}(K_n)_{xy}v_n(y)=v_n(x)$. Together with Lemma~\ref{lem:weighted Schur test}, the equality above gives $\|K_n\|= 1$. Regard $K_n$ and $v_n$ as supported on $Y_n\subseteq X$. Direct calculations show
\begin{equation}\label{eq:coupling-as-coefficient}
 \langle v_n,\chi_AK_n\chi_Bv_n\rangle=\nu_n(A\times B) \quad \text{for any} \quad A,B\subseteq X.
\end{equation}

For $R>0$, set $\alpha(R):=\sup\{\nu_n(A\times B): n\in \NN \text{ and } A,B \subseteq X \text{ with } d(A,B) > R\}$.
Condition \textbf{(Q)} implies that $\alpha(R)\to 0$ as $R\to\infty$. Applying Proposition~\ref{lem:ql-approx} to $T=K_n$ and $\xi=\eta=v_n$ and using \eqref{eq:coupling-as-coefficient}, we obtain
\begin{equation}\label{eq:coupling-from-schur}
\big|\langle v_n,(K_n-M_{\varphi_t}(K_n))v_n\rangle\big|\le4\alpha(R)+4\sqrt{p_t(R)}
\end{equation}
for every $R,t>0$, where $\varphi_t(x,y)=e^{-t\|F(x)-F(y)\|}$ and $p_t(R)=N_X(R) \cdot q_t(\rho_+(R))$. Direct calculations show
\begin{equation}\label{EQ:calculation coupling}
\big|\langle v_n,(K_n-M_{\varphi_t}(K_n))v_n\rangle\big| = \left|1-\sum_{x,y\in Y_n}\varphi_t(x,y)\nu_n(x,y)\right|.
\end{equation}

On the other hand, for every fixed $t>0$ and $S>0$, we have
\begin{align*}
0\le\sum_{x,y\in Y_n}\varphi_t(x,y)\nu_n(x,y)
&\le \nu_n\bigl(\{(x,y):d(x,y)\le S\}\bigr)+e^{-t\rho_-(S)}.
\end{align*}
Condition \textbf{(E)} therefore gives
\[
\limsup_{n\to\infty}\sum_{x,y\in Y_n}\varphi_t(x,y)\nu_n(x,y)
 \le e^{-t\rho_-(S)}.
\]
Letting $S\to\infty$, we obtain
\begin{equation}\label{eq:coupling-kernel-vanishes}
 \lim_{n\to \infty}\sum_{x,y\in Y_n}\varphi_t(x,y)\nu_n(x,y)= 0 \quad \text{for each fixed} \quad t>0.
\end{equation}

Taking $n\to\infty$ in \eqref{eq:coupling-from-schur} and using \eqref{EQ:calculation coupling} and \eqref{eq:coupling-kernel-vanishes}, we obtain
\[
 1\le4\alpha(R)+4\sqrt{p_t(R)}
 \quad \text{for any} \quad R,t>0.
\]
For fixed $R$, we have $p_t(R)\to0$ as $t\to0$ and hence $1\le4\alpha(R)$, which contradicts condition \textbf{(Q)}. Thus, no escaping quasi-local coupling can exist on $X$.
\end{proof}

In the case of coarse disjoint unions of finite metric spaces, Theorem~\ref{thm:esc and ql coupling} can be refined to the following:

\begin{cor}\label{cor:coupling-cdu}
Let $\{(X_n,d_n)\}_{n\in \NN}$ be a sequence of finite metric spaces of uniformly bounded geometry, $X$ be their coarse disjoint union. Then $X$ does not coarsely embed into any Hilbert space \emph{if and only if} there exists a subsequence $\{X_{n_k}\}_k$ and an escaping quasi-local coupling $\{\nu_{n_k}\}_k$ with each $\nu_{n_k}$ supported on $X_{n_k}\times X_{n_k}$.
\end{cor}

\begin{proof}
Assume first that $X$ does not coarsely embed into Hilbert space. By Proposition~\ref{prop:Tessera_generalised_expander}, there is a subsequence $\{X_{n_k}\}_k$ which forms a sequence of generalised expanders. Their witnessing measures form an escaping quasi-local coupling on this subsequence by Lemma~\ref{lem:generalised expander implies coupling}. Conversely, the existence of such a coupling implies that \(X\) does not coarsely embed into Hilbert space by Theorem~\ref{thm:esc and ql coupling}.
\end{proof}

\subsection{Application: ghost quasi-local projections}

Now we use Theorem~\ref{thm:esc and ql coupling} to study compactness of certain ghost operators in the quasi-local algebra. The following is a refined version of Theorem~\ref{introcor:ghost proj are compact}.

\begin{thm}\label{prop:ghost proj are compact}
Let $\{(X_n,d_n)\}_{n}$ be a sequence of finite metric spaces of uniformly bounded geometry, $X$ be their coarse disjoint union, and $\H_0$ be a separable Hilbert space. If $X$ is coarsely embeddable into Hilbert space, then for every locally compact ghost projection \(P\in\B(\ell^2(X;\H_0))\), we have
\[
\operatorname{dist}\big(P,C_q^*(X;\H_0)) = \operatorname{dist}\big(P,C^*(X;\H_0)\big)
=
\begin{cases}
0,& \text{if }P\text{ is compact};\\
\frac12,&\text{if }P\text{ is non-compact}.
\end{cases}
\]
In particular, any ghost projection in $C^*_q(X;\H_0)$ is compact.
%
%any ghost projection in $C^*_q(X;\H_0)$ is compact. Moreover, for any non-compact locally compact ghost projection $P \in \B(\ell^2(X;\H_0))$, we have
%\[
%\mathrm{dist}\left(P,C^*_{q}(X;\H_0)\right)=\mathrm{dist}\big(P,C^*(X;\H_0)\big)=\frac12.
%\]
\end{thm}

\begin{rem}
The Roe-algebra analogue of the compactness conclusion of Theorem~\ref{prop:ghost proj are compact} is already known. Indeed, if \(X\) is coarsely embeddable, then the inclusion $\KK(\ell^2(X;\H_0)) \hookrightarrow G(X;\H_0)$ induces an isomorphism on \(K\)-theory by \cite[Corollary 36]{Fin14} (see also \cite[Theorem C]{WFZ25}). Since \(X\) is sparse, then by \cite[Lemma 4.4]{GLW26}, every ghost projection in \(C^*(X;\H_0)\) is compact.
Theorem~\ref{prop:ghost proj are compact} strengthens this picture to the quasi-local algebra. This does not follow from the preceding \(K\)-theoretic argument, since it is still open whether $C^*(X)$ and $C_q^*(X)$ have the same \(K\)-theory for coarsely embeddable spaces.
\end{rem}

%Before proving Proposition~\ref{prop:ghost proj are compact}, we record two facts about projections.

To prove Theorem~\ref{prop:ghost proj are compact}, we start with some elementary facts on ghost projections.

\begin{lem}\label{lem:ghost projeciton elementary}
Let \(X\) be a discrete metric space of bounded geometry, and \(\mathcal H_0\) be a separable Hilbert space.
\begin{enumerate}
 \item If \(P\in\B(\ell^2(X;\H_0))\) is a ghost projection and \(Q\) is a quasi-local projection with \(\|P-Q\|<1\), then \(Q\) is a ghost.
 \item For every locally compact projection \(P\in\B(\ell^2(X;\H_0))\), we have
\[
   \operatorname{dist}\big(P,C^*(X;\H_0)\big)\le\frac12.
\]
\end{enumerate}
\end{lem}

\begin{proof}
(1). Since \(P\) is a ghost projection, for each fixed \(R>0\) we have
\[
\|P\chi_{B(x,R)}\|^2 \le \sum_{y\in B(x,R)}\|P\chi_{\{y\}}\|^2 =\sum_{y\in B(x,R)}\|P_{yy}\|\le N_X(R) \cdot \sup_{y\in B(x,R)}\|P_{yy}\| \to 0
\]
as $x\to \infty$. Given \(\varepsilon>0\), quasi-locality of \(Q\) provides \(R>0\) such that
\[
\|\chi_{X\setminus B(x,R)}Q\chi_{\{x\}}\|<\varepsilon \quad \text{for any} \quad x\in X.
\]
Hence, $\|PQ\chi_{\{x\}}\| \le \|P\chi_{B(x,R)}\|+\varepsilon$ and hence, \(\|PQ\chi_{\{x\}}\|\to0\) as \(x\to\infty\). Note that
\[
\|Q\chi_{\{x\}}\| \le \|(Q-P)Q\chi_{\{x\}}\|+\|PQ\chi_{\{x\}}\|\le \|P-Q\| \cdot \|Q\chi_{\{x\}}\|+\|PQ\chi_{\{x\}}\|.
\]
Since $\|P-Q\|<1$, the above implies \(\|Q\chi_{\{x\}}\|\to 0\) as \(x\to\infty\). Since $Q$ is a projection, we have $\|Q_{xy}\| \le \|Q\chi_{\{x\}}\|\cdot\|Q\chi_{\{y\}}\|$. This concludes that \(Q\) is a ghost.

(2). Fix $\varepsilon>0$. Since $P$ is locally compact, we can choose a diagonal operator $e:=\mathrm{(SOT)}$-$\sum_{x\in X}e_{x}$, where $e_x$ is a finite-rank projection on $\ell^2(\CC \delta_x; \H_0) \cong \H_0$, such that $\|(1-e)P\|<\varepsilon$. The projection \(e\) is locally compact and has propagation zero, so \(e\in C^*(X;\H_0)\). Moreover, we have
\[
\|P-ePe\| \le \|(1-e)P\|+\|eP(1-e)\| <2\varepsilon.
\]
Since \(0\le ePe\le e\), we have $\left\|ePe-\frac12e\right\|\le\frac12$. Therefore,
\[
\operatorname{dist}\bigl(P,C^*(X;\mathcal H_0)\bigr) \le \left\|P-\frac12e\right\| <\frac12+2\varepsilon.
\]
Letting $\varepsilon \to 0$, we conclude (2).
\end{proof}

Consequently, we obtain the following:

\begin{prop}\label{prop:general ghost compact distance}
Let \(X\) be a discrete metric space of bounded geometry, and let \(\mathcal H_0\) be a separable Hilbert space. Then the following are equivalent:
\begin{enumerate}
 \item Every ghost projection in \(C_q^*(X;\mathcal H_0)\) is compact;
 \item Every ghost projection \(P\in\B(\ell^2(X;\H_0))\) with $\operatorname{dist}(P,C_q^*(X;\H_0))<\frac12$ is compact;
 \item Every locally compact ghost projection \(P\in\B(\ell^2(X;\H_0))\) satisfies:
 \[
\operatorname{dist}\big(P,C_q^*(X;\H_0)) = \operatorname{dist}\big(P,C^*(X;\H_0)\big)
=
\begin{cases}
0,& \text{if }P\text{ is compact};\\
\frac12,&\text{if }P\text{ is non-compact}.
\end{cases}
\]
\end{enumerate}
\end{prop}

\begin{proof}
``(1) $\Rightarrow$ (2)'': Let \(P\in\B(\ell^2(X;\H_0))\) be a ghost projection which satisfies $\operatorname{dist}(P,C_q^*(X;\H_0))<\frac12$. Choose self-adjoint \(T\in C_q^*(X;\H_0)\) such that $\delta:=\|P-T\|<\frac12$. Since \(P\) is a projection, we have $\sigma(T)\subseteq[-\delta,\delta]\cup[1-\delta,1+\delta]$. Hence, functional calculus provides a projection $Q:=\chi_{(1/2,\infty)}(T)\in C_q^*(X;\H_0)$ satisfying
\[
\|P-Q\|
\le\|P-T\|+\|T-Q\|
\le2\delta<1.
\]
By Lemma~\ref{lem:ghost projeciton elementary}, \(Q\) is a ghost projection and hence compact by (1). Hence, $Q$ has finite rank. Since \(\|P-Q\|<1\), the projection \(P\) must also have finite rank.

``(2) $\Rightarrow$ (3)'': Let \(P\in\B(\ell^2(X;\H_0))\) be a locally compact ghost projection. If \(P\) is compact, then \(P\in C^*(X;\H_0)\), so both distances are zero. Now suppose that \(P\) is non-compact. By (2), we have $\operatorname{dist}(P,C_q^*(X;\H_0))\geq \frac12$. Combining with Lemma~\ref{lem:ghost projeciton elementary}(2), we conclude (3).
Finally, ``(3) $\Rightarrow$ (1)'' is trivial.
\end{proof}

To verify condition (1) in Proposition~\ref{prop:general ghost compact distance} for coarsely embeddable sparse spaces, we will use Theorem~\ref{thm:esc and ql coupling}. Let $\H$ be a Hilbert space and $T \in \B(\H)$. If $T$ is trace class, denote by $\Tr(T)$ the trace of $T$ and by $\|T\|_1$ the trace norm of $T$. If $T$ is Hilbert-Schmidt, denote $\|T\|_2$ the Hilbert-Schmidt norm of $T$.

%We introduce the following result

\begin{prop}\label{prop:traceratio}
Let $\{(X_n,d_n)\}_{n\in \NN}$ be a sequence of finite metric spaces of uniformly bounded geometry, $X$ be their coarse disjoint union, and $\H_0$ be a separable Hilbert space. Assume $X$ coarsely embeds into Hilbert space. Let $T:=\mathrm{(SOT)}$-$\sum_{n}T_n \in C^*_q(X;\H_0)$ be a positive ghost operator such that each $T_n$ belongs to the trace class. If
\begin{equation}\label{eq:trbound}
 \sup_{n:T_n\neq 0} \frac{\Tr(T_n)}{\Tr(T_n^2)}<\infty,
\end{equation}
then only finitely many $T_n$ are non-zero. In particular, $T$ is compact.
\end{prop}

To prove Proposition~\ref{prop:traceratio}, we need the following auxiliary result.

\begin{lem}\label{lem:trace coupling}
Let $(Y,d)$ be a finite metric space and $\H_0$ be a separable Hilbert space. For a non-zero positive trace-class operator $T \in \B(\ell^2(Y;\H_0))$, set
\[
c_T=\frac{\Tr(T)}{\Tr(T^2)} \quad \text{and} \quad \nu_T(x,y)=\frac{\|T_{xy}\|_{2}^2}{\Tr(T^2)} \quad \text{for} \quad x,y\in Y.
\]
Then $\nu_T$ is a symmetric probability measure. Moreover, we have:
\begin{enumerate}
 \item For any $A,B \subseteq Y$, $\nu_T(A\times B) \leq c_T \cdot \|\chi_AT\chi_B\|$;
 \item For any $R \geq 0$,
 \[
 \nu_T(\{(x,y) \in Y^2: d(x,y) \leq R\}) \leq c_T \cdot \sup_{x\in Y}|B(x,R)| \cdot \sup_{x\in Y}\|T_{xx}\|.
 \]
\end{enumerate}
\end{lem}

\begin{proof}
Since $T$ is a nonzero positive trace-class operator, $0<\Tr(T^2)<\infty$. Also
\[
 \sum_{x,y\in Y}\|T_{xy}\|_2^2=\|T\|_2^2=\Tr(T^2),
\]
and $T_{yx}=T_{xy}^*$ gives the symmetry of $\nu_T$. For any $A,B \subseteq Y$, both $\chi_A T^{1/2}$ and $\chi_B T^{1/2}$ are Hilbert-Schmidt. Then we have
\[
\|\chi_A T \chi_B\|_1 \leq \|\chi_A T^{1/2}\|_2\cdot \|T^{1/2}\chi_B\|_2 =\sqrt{\Tr(\chi_A T \chi_A)\Tr(\chi_B T \chi_B)}\leq\Tr(T).
\]
Moreover, recall that $\|S\|_2^2 \leq \|S\| \cdot \|S\|_1$ for any trace class operator $S$. Hence,
\[
\nu_T(A\times B)=\frac{\|\chi_AT\chi_B\|_2^2}{\Tr(T^2)} \leq\frac{\Tr(T)}{\Tr(T^2)} \|\chi_AT\chi_B\|,
\]
which concludes (1).

As for (2), let $J_x:\H_0\to \ell^2(Y;\H_0)$ be the inclusion at $x \in Y$, and put $B_x=T^{1/2}J_x$. Then $T_{xy}=B_x^*B_y$ and $\|B_x\|^2=\|T_{xx}\|$. Hence
\[
 T_{xy}^*T_{xy}
 =B_y^*B_xB_x^*B_y
 \le\|T_{xx}\|\,B_y^*B_y
 =\|T_{xx}\|\,T_{yy}.
\]
Therefore, we obtain
\begin{align*}
 \sum_{d(x,y)\le R}\|T_{xy}\|_2^2 &=\sum_{d(x,y)\le R}\Tr(T_{xy}^*T_{xy}) \leq\sup_{x\in Y}\|T_{xx}\| \cdot \sum_{y\in Y}|B_Y(y,R)|\cdot \Tr(T_{yy}) \\
 &\leq\sup_{x\in Y}\|T_{xx}\| \cdot \sup_{x\in Y}|B(x,R)| \cdot \Tr(T).
\end{align*}
Dividing by $\Tr(T^2)$, we conclude (2).
\end{proof}

\begin{proof}[Proof of Proposition~\ref{prop:traceratio}]
Assume that there are infinitely many $n\in \NN$ such that $T_n \neq 0$, and we pass to that subsequence. Denote $C:=\sup_{n} \Tr(T_n)/\Tr(T_n^2)$, and let $\nu_n=\nu_{T_n}$ be the measures from Lemma~\ref{lem:trace coupling}.

Given $\varepsilon>0$, there exists $R>0$ such that $\|\chi_A T \chi_B\|<\varepsilon/C$ whenever $d(A,B)>R$. For $A,B\subseteq X$ with $d(A,B)>R$, Lemma~\ref{lem:trace coupling}(1) gives
\[
 \nu_n(A\times B) = \nu_n\left((A \cap X_n)\times (B \cap X_n)\right) \le C\|\chi_{A \cap X_n}T_n\chi_{B \cap X_n}\|
 =C\|\chi_{A \cap X_n}T\chi_{B \cap X_n}\|<\varepsilon.
\]
On the other hand, Lemma~\ref{lem:trace coupling}(2) gives
\[
 \nu_n(\{(x,y) \in X_n^2: d(x,y) \leq S\})\leq C\cdot \sup_{x\in X_n} |B(x,S)| \cdot \sup_{x\in X_n} \|(T_n)_{xx}\|,
\]
which goes to $0$ since $T$ is a ghost. Hence, $\nu_n$ provides an escaping quasi-local coupling. This contradicts the coarse embeddability of $X$ by Theorem~\ref{thm:esc and ql coupling}.
\end{proof}

%\begin{cor}\label{cor:ghost proj are compact}
%Let $\{(X_n,d_n)\}_{n\in \NN}$ be a sequence of finite metric spaces of uniformly bounded geometry, $X$ be their coarse disjoint union, and $\H_0$ be a Hilbert space. If $X$ can be coarsely embedded into Hilbert space, then any ghost projection in $C^*_q(X;\H_0)$ is compact.
%\end{cor}

Now we prove Theorem~\ref{prop:ghost proj are compact}.

\begin{proof}[Proof of Theorem~\ref{prop:ghost proj are compact}]
Thanks to Proposition~\ref{prop:general ghost compact distance}, it suffices to show that every ghost projection in \(C_q^*(X;\mathcal H_0)\) is compact. We may assume that \(\H_0\) is infinite-dimensional by enlarging the coefficient space if necessary.
Fix a ghost projection $P \in C^*_q(X;\H_0)$, and set $A:=\mathrm{(SOT)}$-$\sum_{n} \chi_{X_n} P \chi_{X_n}$. Simply denote $A_n:=\chi_{X_n} P \chi_{X_n}$ for each $n$, which is a compact positive contraction. It follows from \cite[Lemma 3.14]{BCZ23} that $P-A$ is compact. Since $P$ is a projection, $A$ is self-adjoint and $A^2 - A$ is compact, which implies that $\|A_n^2 - A_n\| \to 0$.

For each $n\in \NN$, define $Q_n:=\chi_{[\frac{1}{2},1]}(A_n)$, which is a finite-rank projection. From functional calculus, we obtain
\[
\|A_n - Q_n\| \leq 2\|A_n - A_n^2\| \to 0\quad \text{as} \quad  n\to \infty.
\]
Setting $Q:=\mathrm{(SOT)}$-$\sum_{n}Q_n$, we obtain that $Q - A$ is compact and hence, $Q - P$ is also compact. Moreover, $Q$ is a ghost projection. Note that $\Tr(Q_n^2) = \Tr(Q_n)$ for each $n$. Proposition~\ref{prop:traceratio} implies that $Q$ is compact; and hence $P$ is compact.
\end{proof}

\subsection{A non-embeddable example}

We provide an example showing that the converse to Theorem~\ref{prop:ghost proj are compact} does not hold in general. More precisely, we show:

\begin{prop}\label{prop:ghost proj are compact converse}
There exists a uniformly discrete bounded geometry sparse metric space $X$ which does not coarsely embed into any Hilbert space, while every ghost projection in $C^*_q(X;\H_0)$ is compact for any separable Hilbert space $\H_0$.
\end{prop}

\begin{rem}
The Roe-algebra analogue of Proposition~\ref{prop:ghost proj are compact converse} is already known. Indeed, \cite[Proposition 35]{Fin14} shows that $\KK(\ell^2(X;\H_0)) \hookrightarrow G(X;\H_0)$ induces an isomorphism on \(K\)-theory for sparse spaces satisfying the coarse Baum--Connes conjecture and admitting a fibred coarse embedding into Hilbert space. The box space of the free group \(\mathbb{F}_3\) constructed in \cite{DK18} provides such an example: it does not coarsely embed into Hilbert space, but it satisfies the coarse Baum--Connes conjecture by \cite[Corollary 1.2]{DWY23} and admits a fibred coarse embedding by \cite{CWW13}.
\end{rem}

To prove Proposition~\ref{prop:ghost proj are compact converse}, we follow the construction in \cite[Sections 3.1 and 3.2]{DK18}. Using their notation, take $q=29$ and set $\Gamma = \mathbb{F}_3$ to be the free group generated by three elements. For each $n$, take $N_n:=\Gamma(2q^n)$, defined in \cite[Section 3.1]{DK18}. Also denote $\Theta(N_n):=N_n^q[N_n,N_n]$, where $N_n^q$ denotes the subgroup in $\mathbb{F}_3$ generated by $\{g^q: g\in N_n\}$. We equip $\mathbb{F}_3$ with the word length metric, and all the finite quotients with the quotient metrics.
By \cite[Theorem 4.5]{DK18}, there exist two increasing sequences $\{k_i\}_i$ and $\{n_i\}_i$ in $\NN$ such that the coarse disjoint union $X:=\bigsqcup_i \mathbb{F}_3/(N_{n_i} \cap \Theta(N_{k_i}))$ does not coarsely embed into any Hilbert space. We will prove that $X$ provides the example required in Proposition~\ref{prop:ghost proj are compact converse}.

To see this, we need some preliminary results. First, we record the following elementary estimate, closely related to \cite[Corollary 4.3]{ST19}. We include a direct proof in order to keep track of the explicit constant needed below.

\begin{lem}\label{lem:cutoff of quasi-local operators}
Given $T \in C^*_q(X;\H_0)$ and $\delta>0$, choose $R>0$ such that for any $A,B \subseteq X$ with $d(A,B) > R$, then $\|\chi_A T \chi_B\| < \delta$. Let $F$ be a finite subset of $X$ with $d(x,x') > R$ for any distinct points $x,x' \in F$. Then
\[
  \big\|\chi_F T \chi_F - \sum_{x\in F} \chi_{\{x\}}T\chi_{\{x\}} \big\|\leq 2\delta.
\]
\end{lem}

\begin{proof}
For a map $\varepsilon: F \to \{-1,1\}$, set $U_\varepsilon:=\sum_{x\in F}\varepsilon(x) \chi_{\{x\}} \in \B(\ell^2(F;\H_0))$. Averaging over all choices of \(\varepsilon\) gives
\begin{equation*}\label{EQ:averaging}
\frac1{2^{|F|}}
\sum_{\varepsilon\in\{-1,1\}^{F}}
U_\varepsilon(\chi_FT\chi_F)U_\varepsilon
=
\sum_{x\in F}\chi_{\{x\}}T\chi_{\{x\}}
\quad \text{since} \quad
\frac1{2^{|F|}}
\sum_{\varepsilon}
\varepsilon(x)\varepsilon(y) =
\begin{cases}
0,& \text{if } x\neq y;\\
 1, & \text{if } x=y.
\end{cases}
\end{equation*}

Fixing a map \(\varepsilon\), set $F_+:=\{x\in F:\varepsilon(x)=1\}$ and $F_-:=\{x\in F:\varepsilon(x)=-1\}$. Then
\[
\|\chi_FT\chi_F-
U_\varepsilon(\chi_FT\chi_F)U_\varepsilon\|
= 2 \|\chi_{F_+}T\chi_{F_-} + \chi_{F_-}T\chi_{F_+}\| \leq 2\delta.
\]
Combining them together, we obtain
\[
\big\| \chi_FT\chi_F - \sum_{x\in F}\chi_{\{x\}}T\chi_{\{x\}} \big\| =
\big\|\frac1{2^{|F|}}\sum_{\varepsilon}\big(\chi_FT\chi_F-U_\varepsilon(\chi_FT\chi_F)U_\varepsilon\big)\big\|
\leq 2\delta.
\]
This finishes the proof.
\end{proof}

The following result is inspired by \cite[Proposition 6.4]{KLVZ21}.

\begin{lem}\label{lem:covering maps ghosts to ghosts}
Let $\{X_n\}_n, \{Y_n\}_n$ be two sequences of finite metric spaces of uniformly bounded geometry, and $(X,d_X), (Y,d_Y)$ be their coarse disjoint unions, respectively. For each $n$, let $\pi_n: X_n \to Y_n$ be a surjective map, and set $\pi:=\bigsqcup_n \pi_n: X \to Y$. Assume that $\pi$ is coarse. Let $\H_0$ be an infinite-dimensional separable Hilbert space, and $V: \ell^2(X;\H_0) \to \ell^2(Y;\H_0)$ be an isometry covering $\pi$. Then:
\begin{enumerate}
 \item For any $T \in C^*_q(X;\H_0)$, we have $\Ad_V(T) \in C^*_q(Y;\H_0)$;
 \item For each $n\in \NN$, denote $s_n:=\inf\{d_X(x,x'):x\neq x' \text{ in } X_n \text{ with } \pi_n(x)=\pi_n(x')\}$, where $\inf\emptyset=+\infty$. If $s_n \to \infty$, then $\Ad_V(T) \in G_q(Y;\H_0)$ for any $T \in G_q(X;\H_0)$.
\end{enumerate}
\end{lem}

\begin{proof}
(1) comes directly from \cite[Proposition 2.9]{BCZ23}, so we focus on (2). Let $T$ be a ghost in $C^*_q(X;\H_0)$, and set $S:=\Ad_V(T) = VTV^*$.

Take $C\geq 0$ such that $\supp(V) \subseteq \{(y,x)\in Y\times X:d_Y(y,\pi(x))\leq C\}$.
For \(y\in Y\), put $E_y:=\pi^{-1}(B_Y(y,C))$. Then for each $y\in Y$, we have
\begin{equation}\label{EQ:chi y V}
\chi_{\{y\}}V=\chi_{\{y\}}V\chi_{E_y}.
\end{equation}
Given $\varepsilon>0$, choose $R>0$ such that for any $A,B \subseteq X$ with $d(A,B) > R$, we have $\|\chi_A TT^* \chi_B\| < \frac{\varepsilon^2}{4N_Y(C)^2}$. Since $G_q(X;\H_0)$ is an ideal, $TT^*$ is a ghost. Hence there is a finite subset $K \subseteq X$ such that
\[
\|(TT^*)_{xx'}\| \leq \frac{\varepsilon^2}{2N_Y(C)^2} \quad  \text{for any} \quad (x,x') \notin K \times K.
\]
Choose $N_0 \in \NN$ such that for any $n > N_0$, we have $X_n \cap K = \emptyset$, $s_n > R$ and for any $m\neq n$, $d_X(X_m, X_n) > R$ and $d_Y(Y_m, Y_n) > C$.

Claim: For any $n > N_0$, any $y\in Y_n$ and $z\in Y$, we have $\|S_{yz}\| \leq \varepsilon$. To see this, fix $y\in Y_n$ and $z\in Y$. From \eqref{EQ:chi y V}, we have
\[
\|\chi_{\{y\}}VT\| = \|\chi_{\{y\}}V\chi_{E_y}T\| \leq \sum_{w\in B_Y(y,C)} \|\chi_{\{y\}}V\chi_{\pi^{-1}(w)}T\| \leq N_Y(C) \cdot \sup_{w\in B_Y(y,C)}\|\chi_{\pi^{-1}(w)}T\|.
\]
Since $d_Y(Y_m, Y_n) > C$ for any $m \neq n$, we obtain $E_y \subseteq X_n$. Hence for each $w\in B_Y(y,C)$, different points in $\pi^{-1}(w)$ have distance at least $s_n$. By Lemma~\ref{lem:cutoff of quasi-local operators}, we have
\begin{align*}
  \|\chi_{\pi^{-1}(w)} TT^* \chi_{\pi^{-1}(w)}\| &\leq \left\|\chi_{\pi^{-1}(w)} TT^* \chi_{\pi^{-1}(w)}- \sum_{x\in \pi^{-1}(w)} \chi_{\{x\}}TT^*\chi_{\{x\}} \right\| + \sup_{x\in \pi^{-1}(w)} \|(TT^*)_{xx}\| \\
  &\leq 2 \cdot \frac{\varepsilon^2}{4N_Y(C)^2} + \frac{\varepsilon^2}{2N_Y(C)^2} = \frac{\varepsilon^2}{N_Y(C)^2}.
\end{align*}
Therefore, we obtain
\[
\|S_{yz}\| = \|\chi_{\{y\}} VTV^* \chi_{\{z\}}\| \leq \|\chi_{\{y\}}VT\| \leq N_Y(C) \cdot \frac{\varepsilon}{N_Y(C)} = \varepsilon.
\]
This proves the claim. Applying the same argument to $T^*$, we conclude that $S$ is ghost. This finishes the proof.
\end{proof}

%The case we are interested in comes from box spaces of groups as follows:
%
%\begin{lem}\label{lem:box spaces fibration}
%Let $\Gamma$ be a finitely generated group, equipped with a left-invariant word length metric. Let $\{N_n\}_n, \{H_n\}_n$ be two  decreasing sequences of finite-index normal subgroups in $\Gamma$ such that $\bigcap_n N_n = \{1\}$ and $H_n \subseteq N_n$ for each $n$. Equip $\Gamma/N_n$ and $\Gamma/H_n$ with the quotient metrics. Let $X$ be the coarse disjoint union of $\{\Gamma/H_n\}_n$, and $Y$ the coarse disjoint union of $\{\Gamma/N_n\}_n$. Denote $\pi_n: \Gamma/H_n \to \Gamma/N_n$ the canonical quotient map, and set $\pi:=\bigsqcup_n \pi_n: X \to Y$. Then $\pi$ is coarse and $s_n:=\inf\{d_X(x,x'):x\neq x' \text{ in } \Gamma/H_n \text{ with } \pi_n(x)=\pi_n(x')\}$ goes to infinity as $n\to \infty$.
%\end{lem}
%
%\begin{proof}
%Since each $\pi_n$ is the canonical quotient map, it is contractive. By the definition of coarse disjoint unions, it is clear that $\pi$ is coarse.
%
%Fixing $R>0$, \cite[Proposition 4.4.3]{NY12} provides $N\in \NN$ such that for any $n> N$, the canonical quotient map $q_n: \Gamma \to \Gamma/N_n$ is isometric on $B_\Gamma(1,R)$. Since $q_n = \pi_n \circ p_n$ for the canonical quotient map $p_n: \Gamma \to \Gamma/H_n$ and $\pi_n$ and $p_n$ are contractive, it follows that $\pi_n$ is isometric on $B_{\Gamma/H_n}(1,R)$. By left invariance of the quotient metrics, this implies $s_n > R$. Since \(R>0\) is arbitrary, we obtain \(s_n\to\infty\).
%\end{proof}

Now we are ready to prove Proposition~\ref{prop:ghost proj are compact converse}.

\begin{proof}[Proof of Proposition~\ref{prop:ghost proj are compact converse}]
Take $X=\bigsqcup_i \mathbb{F}_3/(N_{n_i} \cap \Theta(N_{k_i}))$, which does not coarsely embed into any Hilbert space by \cite[Theorem 4.5]{DK18}. On the other hand, \cite[Theorem 3.9]{DK18} (cf. \cite[Theorem 4]{Khu14}) shows that the coarse disjoint union $Y:=\bigsqcup_i \mathbb{F}_3/\Theta(N_{k_i})$ can be coarsely embedded into Hilbert space. Denote by $\pi: X \to Y$ the map induced by the canonical quotients $\pi_i: \mathbb{F}_3/(N_{n_i} \cap \Theta(N_{k_i})) \to \mathbb{F}_3/\Theta(N_{k_i})$, which is coarse by the definition of coarse disjoint unions.

Fixing $R>0$, \cite[Proposition 4.4.3]{NY12} provides $i_0\in \NN$ such that for any $i> i_0$, the canonical quotient map $q_i: \mathbb{F}_3 \to \mathbb{F}_3/\Theta(N_{k_i})$ is isometric on $B_{\mathbb{F}_3}(1,R)$. Since $q_i = \pi_i \circ p_i$ for the canonical quotient map $p_i: \mathbb{F}_3 \to \mathbb{F}_3/(N_{n_i} \cap \Theta(N_{k_i}))$ and $\pi_i$ and $p_i$ are contractive, it follows that $\pi_i$ is isometric on $B_{\mathbb{F}_3/(N_{n_i} \cap \Theta(N_{k_i}))}(1,R)$. By the left invariance of the quotient metrics, this implies $s_i > R$ for
\[
s_i:=\inf \left\{d_X(x,x'):x\neq x' \text{ in } \mathbb{F}_3/(N_{n_i} \cap \Theta(N_{k_i})) \text{ with } \pi_i(x)=\pi_i(x') \right\}.
\]
Since \(R>0\) is arbitrary, we obtain \(s_i\to\infty\).

%Lemma~\ref{lem:box spaces fibration} can be applied to $\pi$ and hence, $\pi$ is coarse and the corresponding fibre separation constants \(s_i\) tend to infinity.

Now fix an infinite-dimensional separable Hilbert space $\H_0$, and a ghost projection $P$ in $C^*_q(X;\H_0)$. Choose an isometry $V:\ell^2(X;\H_0) \to \ell^2(Y;\H_0)$ covering $\pi$, whose existence can be guaranteed by \cite[Proposition 4.3.4]{WY20}. According to Lemma~\ref{lem:covering maps ghosts to ghosts}, $\Ad_V(P)$ is a ghost quasi-local projection in $\B(\ell^2(Y;\H_0))$. Theorem~\ref{prop:ghost proj are compact} shows that $\Ad_V(P)$ is compact, so $P=V^* \Ad_V(P) V$ is compact.

Finally, we consider finite-dimensional $\H_0$. Embedding it isometrically into an infinite-dimensional separable Hilbert space \(\widetilde \H_0\), a ghost projection \(T\) in $\B(\ell^2(X;\H_0))$ gives a ghost projection \(T\oplus0\) in \(\B(\ell^2(X;\widetilde \H_0))\), and compactness of \(T\oplus0\) implies compactness of \(T\). This finishes the proof.
\end{proof}

\section{Characterisation II: Schur multipliers and ghost operators}\label{sec:schur-ghost}

This section is devoted to Theorem~\ref{thm:char by multiplier} and its applications.

For a metric space \((X,d)\), recall that $\Pz(X)$ is the set of all Schur multiplier kernels which are normalised positive-definite and satisfy
\[
 \lim_{R\to\infty}\sup_{d(x,y)>R}|\varphi(x,y)|=0.
\]
For $\varphi \in \Pz(X)$, its Schur multiplier $M_{\varphi}$ is a unital completely positive contraction.
Also recall that
\[
J_0(X)=\{T \in C^*_u(X):M_\varphi(T^*T)\in\KK(\ell^2(X)) \text{ for every } \varphi\in\Pz(X)\}.
\]

A subset $K \subseteq X \times X$ is called a \emph{partial translation} if there exists $S>0$ such that $K \subseteq E_S$ and for each $x\in X$, there is at most one $y\in X$ with $(x,y) \in K$ and at most one $z\in X$ with $(z,x) \in K$. In this case, there is a bijection $t: D \to R$ with $D,R \subseteq X$ such that $K$ coincides with $\operatorname{graph}(t)=\{(t(y),y):y\in D\}$.
%\begin{thm}\label{thm:char by multiplier}
%Let $(X,d)$ be a discrete metric space of bounded geometry. The following are equivalent:
%\begin{enumerate}
%\item The space $X$ can be coarsely embedded into some Hilbert space;
%\item There exists a sequence $\{\varphi_n\}_{n\in\mathbb N}\subseteq \Pz(X)$ such that $\|M_{\varphi_n}(T)-T\| \to 0$ as $n\to \infty$, for any $T \in C^*_{uq}(X)$;
%\item Any positive $T\in C^*_u(X)$ with $M_\varphi(T)\in \KK(\ell^2(X))$ for any $\varphi\in\Pz(X)$, is compact;
%\item Any positive $T\in C^*_{uq}(X)$ with $M_\varphi(T)\in \KK(\ell^2(X))$ for any $\varphi\in\Pz(X)$, is compact;
%\item $J_0(X)=\KK(\ell^2(X))$.
%\end{enumerate}
%\end{thm}

The main assertion of Theorem~\ref{thm:char by multiplier} is that compactness of operators in $J_0(X)$ characterises coarse embeddability into Hilbert space. We package the proof into the following subsection.

\subsection{Proof of the characterisation}

We start with the following algebraic assertion concerning $J_0(X)$.

\begin{lem}\label{lem:SC implies ghost}
Let $(X,d)$ be a discrete metric space of bounded geometry. Then $J_0(X)$ is a closed two-sided ideal of $C^*_u(X)$ and $\KK(\ell^2(X)) \subseteq J_0(X) \subseteq G_u(X)$. Moreover, for a positive operator $T\in C^*_u(X)$, $T \in J_0(X)$ \emph{if and only if} $M_\varphi(T)\in \KK(\ell^2(X))$ for any $\varphi\in\Pz(X)$.
\end{lem}

\begin{proof}
Since the Schur multiplier $M_\varphi$ is linear, positive and contractive, it is routine to check that $J_0(X)$ is a norm-closed linear subspace of $C^*_u(X)$. For example, for any $T_1,T_2\in J_0(X)$ and $\varphi\in\Pz(X)$, positivity of $M_\varphi$ gives
\[
 0\le M_\varphi((T_1+T_2)^*(T_1+T_2)) \le2 M_\varphi(T_1^*T_1)+ 2 M_\varphi(T_2^*T_2)\in\KK(\ell^2(X)).
\]
Since $\KK(\ell^2(X))$ is hereditary, we obtain $T_1+T_2 \in J_0(X)$.

Now we check that $J_0(X)$ is an ideal in $C^*_u(X)$. For $T \in J_0(X)$ and $S \in C^*_u(X)$, we have $0\le M_\varphi ((ST)^*(ST))\le\|S\|^2 M_\varphi(T^*T)$, which proves that $ST \in J_0(X)$. To see that $TS \in J_0(X)$, first assume that $S$ is a diagonal operator. Then
\[
M_\varphi((TS)^*(TS))=M_\varphi(S^*T^*TS)=S^*M_\varphi(T^*T)S\in\KK(\ell^2(X)).
\]
Next, we assume that $S=\chi_K$ for some partial translation $K$ and let $t: D \to R$ be the bijection, where $D,R \subseteq X$, such that $S$ is the graph of $t$. Define a kernel $\varphi'$ on $X \times X$ by
\[
 \varphi'(x,y)=
 \begin{cases}
 \varphi(t^{-1}(x), t^{-1}(y)), & \text{if } x,y\in R,\\
 \delta_{x=y}, & \text{otherwise}.
 \end{cases}
\]
Here $\delta_{x=y}$ equals $1$ if $x=y$, and $0$ if $x\neq y$. It is routine to check $\varphi' \in \Pz(X)$ and
\[
M_{\varphi}((TS)^*(TS)) =S^*M_{\varphi'}(T^*T)S \in \KK(\ell^2(X)).
\]
Any finite propagation operator has the form $S=\sum_{j=1}^n f_j V_{t_j}$ for $f_j \in \ell^\infty(X)$ and bijections $t_j$ as above. Taking the norm closure, we obtain $TS \in J_0(X)$ for any $S \in C^*_u(X)$.

Since $M_\varphi$ preserves $\KK(\ell^2(X))$ for each $\varphi \in \Pz(X)$, we have $\KK(\ell^2(X)) \subseteq J_0(X)$. Take $\Delta:=\{(x,x): x\in X\}$ and $\varphi = \chi_\Delta$, which is in $\Pz(X)$. Given $T \in J_0(X)$, we have $M_\varphi(T^*T)_{xx} = \|T\delta_x\|^2$ and $M_\varphi(TT^*)_{xx} = \|T^*\delta_x\|^2$. Compactness shows that $\|T\delta_x\| \to 0$ and $\|T^*\delta_x\| \to 0$ as $x\to \infty$, where the second follows from $T^* \in J_0(X)$. Since $|A_{xy}|\leq \min\left\{\|A\delta_y\|,\|A^*\delta_x\|\right\}$ for any $x,y\in X$, we obtain that $T$ is a ghost.

Finally, fix a positive $T \in C^*_u(X)$. Assume $M_\varphi(T)\in \KK(\ell^2(X))$ for any $\varphi\in\Pz(X)$. Since $T^2\le\|T\|T$, we obtain that $M_\varphi(T^2) \leq \|T\|M_\varphi(T) \in \KK(\ell^2(X))$. Hence $T \in J_0(X)$. Conversely, assume that $T \in J_0(X)$. Since $M_\varphi$ is a completely positive contraction, $0\leq M_\varphi(T)^2\leq M_\varphi(T^2)$, which shows that $M_\varphi(T)$ is compact.
\end{proof}

We now combine the matrix obstruction from Section~\ref{sec:obstructions} with the Schur approximation from Section~\ref{sec:cuts} to prove our characterisation.

\begin{proof}[Proof of Theorem~\ref{thm:char by multiplier}]
``(1) $\Rightarrow$ (2)'': Assume that $F:X\to \H$ is a coarse embedding. For $t>0$, put $\varphi_t(x,y):=\exp({-t\|F(x)-F(y)\|})$ for $x,y\in X$. Clearly, $\varphi_t\in\Pz(X)$. Moreover, Corollary~\ref{cor:ql-approx} implies that $\|M_{\varphi_t}(T)-T\|\to 0$ as $t\to 0$.

``(2) $\Rightarrow$ (4)'' and ``(4) $\Rightarrow$ (3)'' hold trivially.

``(3) $\Leftrightarrow$ (5)'': By Lemma~\ref{lem:SC implies ghost}, condition (3) is equivalent to the assertion
that all positive elements in $J_0(X)$ are compact. Since $\KK(\ell^2(X)) \subseteq J_0(X)$ and every element in a $C^*$-algebra is a linear combination of positive elements, we obtain (5).

``(3) $\Rightarrow$ (1)'':
If (1) fails, then there exist $Y=\bigsqcup_nY_n$, $S_n,M_n,C,r_n$ given by Proposition~\ref{prop:matrix-obstruction}. Regard all block-diagonal operators on $\ell^2(Y)$ as operators on $\ell^2(X)$ by zero extension. Set
\[
 S=\mathrm{(SOT)}\text{-}\sum_{n} S_n
 \quad \text{and} \quad
 M=\mathrm{(SOT)}\text{-}\sum_{n} M_n.
\]
Also set
\[
 T_n:=(S_n)_+ \quad \text{and} \quad
 T:=S_+=\mathrm{(SOT)}\text{-}\sum_{n} T_n \in C^*_u(X).
\]
Since $M_n$ is self-adjoint and $\|M_n\|\leq 1$, we have $S_n \leq M_n\leq  \Id_{\ell^2(Y_n)}$, which gives $0 \leq T \leq \Id_{\ell^2(X)}$. Moreover, $\|T_n\|=\lambda_{\max}(S_n)\ge C^{-1}$, so $T$ is not compact.

It remains to show that $M_\varphi(T)$ is compact for any \(\varphi\in\mathcal P_0(X)\). Fix \(\varphi\in\mathcal P_0(X)\), and put \(V:=T^{1/2}\in C_u^*(X)\). Then $0\le T^2=T^{1/2}ST^{1/2}\le V^*MV$.
Choose block-diagonal finite-propagation operators \(V^{(m)}\) such that $\|V^{(m)}-V\|\to 0$.
Fix \(m\), and let \(R:=\ppg(V^{(m)})\). Since \(M_n\) is supported on \(\{(x,y) \in Y_n^2: d(x,y) \geq r_n\}\), $A_n:=(V_n^{(m)})^*M_nV_n^{(m)}$ is supported on \(\{(x,y) \in Y_n^2: d(x,y) \geq r_n - 2R\}\).
Set
\[
\varepsilon_n:= \sup_{d(x,y)\geq r_n-2R} |\varphi(x,y)|, \quad \widetilde V_n=\left(|(V_n^{(m)})_{xy}|\right)_{x,y\in Y_n} \quad \text{and} \quad
K_n:=\widetilde V_n^*M_n\widetilde V_n.
\]
Since \(M_n\) is entrywise non-negative, we have $|(A_n)_{xy}|\le (K_n)_{xy}$ for each $x,y\in Y_n$.
Hence, for any unit vectors \(\xi,\eta \in \ell^2(Y_n)\), we have
\[
|\langle\eta,M_\varphi(A_n)\xi\rangle| \leq \varepsilon_n \langle|\eta|,K_n|\xi|\rangle \leq \varepsilon_n\|K_n\|,
\]
for $|\eta|(x)=|\eta(x)|$ and $|\xi|(x) = |\xi(x)|$. By the Schur test, we have $\|\widetilde V_n\| \leq N_X(R) \cdot \|V_n^{(m)}\|$ and hence, $\|K_n\| \leq N_X(R)^2 \cdot  \|V_n^{(m)}\|^2 \leq N_X(R)^2 \cdot  \|V^{(m)}\|^2$. Therefore,
\[
\|M_\varphi(A_n)\|
\le
\varepsilon_n\|K_n\|
\le
\varepsilon_n \cdot N_X(R)^2 \cdot \|V^{(m)}\|^2 \to 0 \quad \text{as} \quad n \to \infty,
\]
which shows that $M_\varphi\bigl((V^{(m)})^*MV^{(m)}\bigr)$ is compact. Since $M_\varphi$ is contractive, we obtain $M_\varphi(V^*MV) = \lim_{m \to \infty} M_\varphi\bigl((V^{(m)})^*MV^{(m)}\bigr)$, which is also compact. Finally, note that
\[
0\le M_\varphi(T)^2 \le M_\varphi(T^2)\le M_\varphi(V^*MV).
\]
Thus, \(M_\varphi(T)\) is compact. This finishes the proof.
\end{proof}

\subsection{Application: relative Haagerup property}

Let $(X,d)$ be a discrete metric space of bounded geometry. Denote by $E: C^*_{uq}(X) \to \ell^\infty(X)$ the canonical conditional expectation obtained by restricting operators to the diagonal. Equip $C^*_{uq}(X)$ with the $\ell^\infty(X)$-valued inner product
\[
  \langle S, T \rangle_{\ell^\infty(X)}:=E(S^*T) \quad \text{for} \quad S,T \in C^*_{uq}(X).
\]
The corresponding norm is
\begin{equation}\label{eq:Enorm}
  \|T\|_E
  =\|E(T^*T)\|^{1/2}
  =\sup_{y\in X}\Big(\sum_{x\in X}|T_{xy}|^2\Big)^{1/2}
  =\sup_{y\in X}\|T\delta_y\|.
\end{equation}
Clearly, $\|T\|_E \leq \|T\|$. Denote by $\mathcal H_E$ the Hilbert $\ell^\infty(X)$-module completion of $C^*_{uq}(X)$.

\begin{lem}\label{lem:same-module}
The Hilbert module $\mathcal H_E$ is also the completion of $C^*_u(X)$ under $\|\cdot\|_E$.
\end{lem}

\begin{proof}
Let $T\in C^*_{uq}(X)$. For $R>0$, let $T^{(R)} \in \B(\ell^2(X))$ be the restriction of $T$ to $E_R$, which has finite propagation. By \eqref{eq:Enorm}, we have
\[
 \|T-T^{(R)}\|_E
 =\sup_{y\in X}\big\|\chi_{X\setminus B(y,R)}T\delta_y\big\|.
\]
Since $T$ is quasi-local, the right-hand side tends to zero as $R\to\infty$.
\end{proof}

Given $\xi,\eta \in \mathcal{H}_E$, denote the rank-one operator $\theta_{\xi,\eta}$ on $\mathcal{H}_E$ by
\begin{equation}\label{EQ:theta}
\theta_{\xi,\eta}(\zeta)
:=\xi\langle\eta,\zeta\rangle_{\ell^\infty(X)} \quad \text{for} \quad \zeta \in \mathcal{H}_E.
\end{equation}
The compact operators on $\mathcal{H}_E$ are defined to be $\KK(\mathcal H_E):=\overline{\operatorname{span}}
\{\theta_{\xi,\eta}:\xi,\eta\in\mathcal H_E\}$.

We recall the following notion of $E$-Haagerup property from \cite[Definition 5.12]{KLS22}. In our unital setting, the local compactness condition therein is simply compactness on \(\mathcal H_E\) (cf. the proof of \cite[Proposition 5.16]{KLS22}). Thus, in the present setting, the definition takes the following form.

\begin{defn}\label{def:EHaagerup}
We say that $C^*_{uq}(X)$ has the \emph{$E$-Haagerup property with respect to $\ell^\infty(X)$} if there is a net of completely positive maps $\Phi_i:C^*_{uq}(X)\to C^*_{uq}(X)$ such that
\begin{enumerate}
  \item $\Phi_i|_{\ell^\infty(X)}=\Id_{\ell^\infty(X)}$;
  \item $\Phi_i$ can be extended to $\widetilde\Phi_i\in \KK(\mathcal H_E)$;
  \item $\|\Phi_i(T)-T\|_E\to0$ for every $T\in C^*_{uq}(X)$.
\end{enumerate}
\end{defn}

Using Theorem~\ref{thm:char by multiplier}, we manage to prove the following.

\begin{prop}\label{thm:relative_Haagerup}
For a discrete metric space $(X,d)$ of bounded geometry, $X$ can be coarsely embedded into Hilbert space \emph{if and only if} $C^*_{uq}(X)$ has the $E$-Haagerup property with respect to $\ell^\infty(X)$.
\end{prop}

\begin{rem}
Note that the uniform Roe algebra version of Proposition~\ref{thm:relative_Haagerup} was already obtained in \cite[Corollary 5.20]{KLS22}, and here we show that the same characterisation remains valid after replacing \(C_u^*(X)\) by the larger algebra \(C_{uq}^*(X)\).
\end{rem}

To prove Proposition~\ref{thm:relative_Haagerup}, we need the following auxiliary lemma.

\begin{lem}\label{lem:Schur extended to compact}
Let $\phi:X\times X\to\mathbb C$ be a Schur multiplier kernel. Then the associated Schur multiplier $M_\phi$ induces an adjointable operator $\widetilde M_\phi$ on $\mathcal H_E$ with
\[
  \|\widetilde M_\phi\|=\|\phi\|_\infty\le \|\phi\|_S \quad \text{and} \quad
  \widetilde M_\phi^*=\widetilde M_{\overline\phi}.
\]
Moreover, $\widetilde M_\phi\in\KK(\mathcal H_E)$ \emph{if and only if} $\lim\limits_{R\to\infty}\sup_{d(x,y)>R}|\phi(x,y)|=0$.
\end{lem}

\begin{proof}
As we noticed before, the Schur multiplier $M_\phi$ maps $C^*_{uq}(X)$ into itself. Moreover, for $T \in C^*_{uq}(X)$, \eqref{eq:Enorm} gives
\[
 \|M_\phi(T)\|_E^2 =\sup_{y\in X}\sum_{x\in X}|\phi(x,y)|^2|T_{xy}|^2 \le \|\phi\|_\infty^2\|T\|_E^2.
\]
Hence, $M_\phi$ induces a bounded $\ell^\infty(X)$-module operator $\widetilde M_\phi$ on $\mathcal H_E$ and clearly, $\|\widetilde M_\phi\|=\|\phi\|_\infty\le \|\phi\|_S$. Moreover, $\langle M_\phi(S), T\rangle_{\ell^\infty(X)} = \langle S, M_{\overline\phi}(T) \rangle_{\ell^\infty(X)}$ shows that
%\[
%  \langle M_\phi(S), T\rangle_{\ell^\infty(X)} = \langle S, M_{\overline\phi}(T) \rangle_{\ell^\infty(X)} \quad \text{for any} \quad S,T \in C^*_{uq}(X)
%\]
$\widetilde M_\phi$ is adjointable and $\widetilde M_\phi^*=\widetilde M_{\overline\phi}$.

For the ``Moreover'' part, first assume that $\lim\limits_{R\to\infty}\sup_{d(x,y)>R}|\phi(x,y)|=0$. For each $R>0$, denote $\phi_R$ the point-wise multiplication $\phi\cdot \chi_{E_R}$. Since $\chi_{E_R}$ is a Schur multiplier kernel, so is $\phi_R$. Then \eqref{eq:Enorm} shows that
\[
\left\|\widetilde M_\phi-\widetilde M_{\phi_R}\right\| =\|\phi(1-\chi_{E_R})\|_\infty = \sup_{d(x,y)>R}|\phi(x,y)|\to 0 \quad \text{as} \quad R \to +\infty.
\]
For each $R>0$, bounded geometry allows us to decompose $E_R= \bigsqcup_{j=1}^N \gr(t_j)$ for partial translations $t_j: D_j \to R_j$. For each \(j\), set \(V_j\in C_u^*(X)\) and \(a_j\in\ell^\infty(X)\) by
\[
V_j\delta_y=
\begin{cases}
\delta_{t_j(y)},&y\in D_j,\\
0,&y\notin D_j
\end{cases}
\quad \text{and} \quad
a_j(y)=
\begin{cases}
\phi_R(t_j(y),y),&y\in D_j,\\
0,&y\notin D_j.
\end{cases}
\]
Direct calculations show that $\widetilde M_{\phi_R}=\sum_{j=1}^N\theta_{V_ja_j,V_j} \in \KK(\mathcal H_E)$ for $\theta_{V_ja_j,V_j}$ from \eqref{EQ:theta}.
Hence, their limit $\widetilde M_\phi$ belongs to $\KK(\mathcal H_E)$.

Conversely, suppose that $\widetilde M_\phi\in \KK(\mathcal H_E)$. Given $\eta \in \mathcal H_E$, we can approximate $\eta$ in terms of finite propagation operators under the $E$-norm by Lemma~\ref{lem:same-module}. Hence,
\begin{equation}\label{eq:module-vector-decay}
  \lim_{R\to\infty}\sup_{d(x,y)>R}|\eta_{xy}|=0.
\end{equation}
For the matrix unit $e_{xy}\in \B(\ell^2(X))$, we have $\theta_{\xi,\eta}(e_{xy}) = \xi E(\eta^*e_{xy})=\overline{\eta_{xy}}\,\xi e_{yy}$. Hence, $\|\theta_{\xi,\eta}(e_{xy})\|_E \le \|\xi\|_E \cdot |\eta_{xy}|$,
which tends to zero uniformly as $d(x,y)\to\infty$ by \eqref{eq:module-vector-decay}. By finite $\ell^\infty(X)$-linear combinations and taking limits, we obtain
\[
  \lim_{R\to\infty}\sup_{d(x,y)>R}\|K(e_{xy})\|_E=0 \quad \text{for any} \quad K\in\KK(\mathcal H_E).
\]
Taking $K=\widetilde M_\phi$, we have $\|\widetilde M_\phi(e_{xy})\|_E=|\phi(x,y)|$. This proves the converse.
\end{proof}

Finally, we prove Proposition~\ref{thm:relative_Haagerup}.

\begin{proof}[Proof of Proposition~\ref{thm:relative_Haagerup}]
%``$(1)\Rightarrow(3)$'': Assume that $F:X\to\H$ is a coarse embedding to some Hilbert space $\H$. For $t>0$, put $\phi_t(x,y):=\exp(-t\|F(x)-F(y)\|)$. As noticed before, $\phi_t\in \Pz(X)$. Then (3) follows directly from Proposition~\ref{lem:ql-approx}.
%
Assume that $X$ is coarsely embeddable. By Theorem~\ref{thm:char by multiplier}(2), take a sequence $\{\phi_n\}_{n\in\mathbb N}\subseteq \Pz(X)$ such that $\|M_{\phi_n}(T)-T\| \to 0$ for any $T \in C^*_{uq}(X)$. Note that $M_{\phi_n}$ is unital completely positive and fixes $\ell^\infty(X)$. By Lemma~\ref{lem:Schur extended to compact}, we have $\widetilde M_{\phi_n}\in\K(\mathcal \H_E)$. Moreover, $\|S\|_E\le\|S\|$ for any $S \in C^*_{uq}(X)$ implies that $\|M_{\phi_n}(T)-T\|_E \to 0$ for any $T \in C^*_{uq}(X)$.

Conversely, let $\{\Phi_i\}$ be a net satisfying the conditions in Definition~\ref{def:EHaagerup}. Since $\Phi_i$ is ucp and fixes $\ell^\infty(X)$, it is a $\ell^\infty(X)$-bimodule map. For $x,y\in X$, let $e_{xy}$ denote the standard matrix unit.
Since \(e_{xy}=\chi_{\{x\}}e_{xy}\chi_{\{y\}}\), we have
\[
\Phi_i(e_{xy})=\chi_{\{x\}}\Phi_i(e_{xy})\chi_{\{y\}}\in \chi_{\{x\}} C^*_{uq}(X)\chi_{\{y\}}=\CC e_{xy}.
\]
Thus, there is a scalar kernel \(h_i: X \times X \to \CC\) such that $\Phi_i(e_{xy})=h_i(x,y)e_{xy}$. Moreover,
\[
\chi_{\{x\}}\Phi_i(T)\chi_{\{y\}} =\Phi_i(\chi_{\{x\}}T\chi_{\{y\}}) =h_i(x,y)T_{xy}e_{xy} \quad \text{for any} \quad T\in C^*_{uq}(X).
\]
Hence, $\Phi_i(T)_{xy}=h_i(x,y)T_{xy}$ or, in other words, $\Phi_i = M_{h_i}$. By the properties of $\Phi_i$, it is routine to check that $h_i$ is normalised positive-definite. Since $\widetilde \Phi_i$ is compact, Lemma~\ref{lem:Schur extended to compact} shows that for each fixed $i$,
\begin{equation}\label{eq:offdiag-hi}
  \lim_{S\to\infty}\sup_{d(x,y)>S}|h_i(x,y)|=0.
\end{equation}

We now show that $h_i\to1$ uniformly on each $E_R$. Fixing $R>0$, regard $\chi_{E_R}$ as a finite propagation operator in $C^*_u(X)$. By \eqref{eq:Enorm}, we have
\[
\|\Phi_i(\chi_{E_R})-\chi_{E_R}\|_E = \sup_{y\in X}\left(\sum_{x\in B(y,R)}|h_i(x,y)-1|^2\right)^{1/2}\ge
\sup_{d(x,y)\le R}|h_i(x,y)-1|.
\]
Since the left-hand side goes to $0$, we obtain
\[
\lim_i\sup_{d(x,y)\le R}|h_i(x,y)-1|=0
\quad \text{for every} \quad R>0.
\]
By the standard positive-definite-kernel characterisation of coarse embeddability (see, \emph{e.g.}, \cite[Theorem 3.2.8]{Wil09b}), \(X\) coarsely embeds into Hilbert space.
\end{proof}

\subsection{Application: coarse embeddability into $\ell^p$-spaces}

%As another application of Theorem~\ref{thm:char by multiplier}, we now study the relation between coarse embeddability into Hilbert space and into $\ell^p$-spaces for $p\in (2,+\infty)$. Recall from \cite{Nowak2006} that the separable Hilbert space can be coarsely embedded into $\ell^p$-space for any $p\in (2,+\infty)$ (in fact, for any $p \in [1, +\infty]$). So we focus on the converse.

For $2<p<\infty$, a central question in coarse geometry is whether coarse embeddability into $\ell^p$ implies coarse embeddability into Hilbert space for a discrete metric space of bounded geometry.
%whether a discrete metric space of bounded geometry that coarsely embeds into $\ell^p$ can be coarsely embedded into Hilbert space.
The case of $1\le p\le2$ and the converse implication are already known in \cite{Nowak2006}. We provide a partial answer with the help of Theorem~\ref{thm:char by multiplier}.

Recall that a $C^*$-algebra $\mathcal{A}$ has \emph{real rank zero}, denoted by $\RR(\mathcal{A})=0$, if self-adjoint elements with finite spectrum are dense in the set of all self-adjoint elements in $\mathcal{A}$. Also recall that $\mathcal{A}$ has \emph{property $\SP$} if every nonzero hereditary $C^*$-subalgebra of $\mathcal{A}$ contains a non-zero projection.
%It follows from \cite[Theorem 2.6]{BP91} that $RR(\mathcal{A}) = 0$ implies property (SP).
We have the following:

\begin{prop}\label{thm:lp and l2}
Fix $2<p<\infty$, and let $X$ be a discrete metric space of bounded geometry. Then the following are equivalent:
\begin{enumerate}
 \item $X$ coarsely embeds into a Hilbert space;
 \item $X$ coarsely embeds into $\ell^p$ and $\RR(J_0(X))=0$;
 \item $X$ coarsely embeds into $\ell^p$ and $J_0(X)/\KK(\ell^2(X))$ has property $\SP$.
\end{enumerate}
\end{prop}

\begin{rem}\label{rem:RR and SP}
Since \(\RR(\KK(\ell^2(X))) = 0\) and \(K_1(\KK(\ell^2(X)))=0\), \cite[Theorem 3.14 and Proposition 3.15]{BP91} shows that $\RR(J_0(X))=0$ \emph{if and only if} $\RR(J_0(X)/\KK(\ell^2(X))) = 0$.

On the other hand, Lemma~\ref{lem:SC implies ghost} shows that \(\KK(\ell^2(X))\subseteq J_0(X)\). This implies that every non-zero hereditary \(C^*\)-subalgebra of \(J_0(X)\) contains a non-zero rank-one projection, so \(J_0(X)\) always has property $\SP$.
\end{rem}

\begin{proof}[Proof of Proposition~\ref{thm:lp and l2}]
``(1) $\Rightarrow$ (2)'': By Theorem~\ref{thm:char by multiplier}, we have $J_0(X)=\KK(\ell^2(X))$, which has real rank zero.
Moreover, \cite{Nowak2006} shows that a separable Hilbert space coarsely embeds into $\ell^p$. Since $X$ is countable, we conclude (2).

``(2) $\Rightarrow$ (3)'': By Remark~\ref{rem:RR and SP}, $\RR(J_0(X)/\KK(\ell^2(X))) = 0$. Now \cite[Theorem~2.6]{BP91} shows that real rank zero implies property $\SP$, so we obtain (3).

``(3) $\Rightarrow$ (1)'': Suppose that (3) holds but $X$ does not coarsely embed into Hilbert space. By Proposition~\ref{prop:sparse-reduction}, there is a sparse subspace $Y\subseteq X$ which also fails to coarsely embed into Hilbert space. %By assumption, $Y$ coarsely embeds into $\ell^p$.
It is routine to check that $J_0(Y)=\chi_Y J_0(X)\chi_Y$, which implies that
\[
J_0(Y)/\KK(\ell^2(Y))
\cong e\cdot \bigl(J_0(X)/\KK(\ell^2(X))\bigr)\cdot e \quad \text{for} \quad e:=[\chi_Y] \in C_u^*(X)/\KK(\ell^2(X)).
\]
Hence, $B:=J_0(Y)/\KK(\ell^2(Y))$ is a hereditary $C^*$-subalgebra of $J_0(X)/\KK(\ell^2(X))$. Moreover, $B\ne 0$ by
Theorem~\ref{thm:char by multiplier}, since $Y$ does not coarsely embed into Hilbert space.

By assumption, $B$ contains a non-zero projection $q$. Choose a positive contraction $a\in J_0(Y)$ whose image in $B$ is $q$. Since \(a-a^2\) is compact and \(s(1-s)\geq 2/9\) on \([1/3,2/3]\), the spectral mapping theorem implies that \(\sigma(a)\cap[1/3,2/3]\) is finite. Taking $t\in(1/3,2/3)\setminus\sigma(a)$, we define $Q:=\chi_{(t,\|a\|]}(a)$. Then $Q \in J_0(Y)$, whose image in $B$ is $q$ by the naturality of continuous functional calculus.

By Lemma~\ref{lem:SC implies ghost}, \(Q\) is a ghost projection in \(C_u^*(Y)\). Since \(Y\) is sparse and coarsely embeds into \(\ell^p\), \cite[Theorem~2.6 and Lemma~4.4]{GLW26} imply that every ghost projection in \(C^*(Y)\) is compact. Viewing \(C_u^*(Y)\) as a rank-one corner of \(C^*(Y)\), we conclude that \(Q\) is compact. Hence \(q=0\), which is a contradiction.
%By Lemma~\ref{lem:SC implies ghost}, $J_0(Y)\subseteq G_u(Y)$, so $Q$ is a ghost projection in $C_u^*(Y)$. Since $Y$ is sparse and coarsely embeds into $\ell^p$, \cite[Theorem~2.6 and Lemma~4.4]{GLW26} shows that every ghost projection in $C^*(Y)$ is compact. This implies that $Q$ is compact since applying an isometric embedding of $\CC$ into an infinite-dimensional separable Hilbert space identifies $Q$ with a ghost projection in $C^*(Y)$. This shows that $q=0$, which leads to a contradiction.
\end{proof}

\section{Characterisation III: operator norm state localisation}\label{sec:onsl}

Finally, we focus on ONSL and prove Theorem~\ref{thm:CE and ONL}.

%Recall that given a positive trace-class operator $D$ on $\ell^2(X)$ with $\Tr(D)=1$, we have a normal state $\phi_D$ on $\B(\ell^2(X))$ defined by
%\[
%\phi_D: T \mapsto \Tr(TD).
%\]
%It is known that every normal state $\phi$ on $\B(\ell^2(X))$ has this form, and $D$ is called the \emph{density operator} of $\phi$.
%
%In this section, we introduce a weak version of the operator norm localisation property (ONL) to characterise coarse embeddability. The proof uses the two key ingredients developed above: the matrix obstruction of Proposition~\ref{prop:matrix-obstruction} and the Schur approximation from Proposition~\ref{lem:ql-approx}.
%
%\subsection{State localisation and coarse embeddability}

Recall from Definition~\ref{defn:ONSL} that a discrete metric space $(X,d)$ of bounded geometry is said to have ONSL if for any $R>0$ and $\varepsilon \in (0,1)$, there exists a non-increasing function $\omega:[0,\infty)\to [0,1]$ with $\lim_{s\to\infty}\omega(s)=0$ such that for any positive operator $T\in \B(\ell^2(X))$ with $\ppg(T)\le R$, there exists a positive trace-class operator $D$ on $\ell^2(X)$ with $\Tr(D)=1$ and satisfying:
\begin{equation}\label{EQ:TD control}
\Tr(TD)\ge(1-\varepsilon)\|T\|
\end{equation}
and
\begin{equation}\label{EQ:D condition}
|D_{xy}|\le \omega(d(x,y))\sqrt{D_{xx}D_{yy}} \quad \text{for any} \quad x,y\in X.
\end{equation}

%First, recall that given a positive trace-class operator $D$ on $\ell^2(X)$ with $\Tr(D)=1$, we have a normal state $\phi_D$ on $\B(\ell^2(X))$ defined by
%\[
%\phi_D: T \mapsto \Tr(TD).
%\]
%It is known that every normal state $\phi$ on $\B(\ell^2(X))$ has this form, and $D$ is called the \emph{density operator} of $\phi$.

%\begin{defn}\label{defn:ONSL}
%A discrete metric space $(X,d)$ of bounded geometry is said to have the \emph{operator norm state localisation property (ONSL)} if for any $R>0$ and $\varepsilon \in (0,1)$, there exists a non-increasing function $\omega:[0,\infty)\longrightarrow[0,1]$ with $\lim_{s\to\infty}\omega(s)=0$ such that for any positive operator $T\in \B(\ell^2(X))$ with $\ppg(T)\le R$, there exists a positive trace-class operator $D$ on $\ell^2(X)$ with $\Tr(D)=1$ satisfying:
%\begin{equation}\label{EQ:TD control}
%\Tr(TD)\ge(1-\varepsilon)\|T\|
%\end{equation}
%and
%\begin{equation}\label{EQ:D condition}
%|D_{xy}|\le \omega(d(x,y))\sqrt{D_{xx}D_{yy}} \quad \text{for any} \quad x,y\in X.
%\end{equation}
%\end{defn}

%We show that ONSL characterises coarse embeddability.
%
%\begin{thm}\label{thm:CE and ONL}
%Let $(X,d)$ be a discrete metric space of bounded geometry. Then $X$ can be coarsely embedded into some Hilbert space \emph{if and only if} $X$ satisfies the operator norm state localisation property.
%\end{thm}

Theorem~\ref{thm:CE and ONL} claims that ONSL characterises coarse embeddability. Now we present the proof.

\begin{proof}[Proof of Theorem~\ref{thm:CE and ONL}]
Assume first that $F:X\to \H$ is a coarse embedding with control functions $\rho_-,\rho_+$. Fix $R>0$ and $0<\varepsilon<1$. For $t>0$, let
\[
 \varphi_t(x,y):=e^{-t\|F(x)-F(y)\|} \quad \text{for} \quad x,y\in X
\]
as in Proposition~\ref{lem:ql-approx}. Choose $t>0$ sufficiently small such that
\[
 N_X(R)\cdot q_t(\rho_+(R))<\frac{\varepsilon^2}{64},
\]
and put $\omega(s):=e^{-t\rho_-(s)}$. Then $\omega$ is non-increasing and tends to zero at infinity.

Let $0\le T\in B(\ell^2(X))$ have propagation at most $R$, and assume that $T \neq 0$. Since $\chi_AT\chi_B=0$ whenever $d(A,B)>R$, Corollary~\ref{cor:ql-approx} gives $\|M_{\varphi_t}(T)-T\|<\frac\varepsilon2\|T\|$. Choose a unit vector $\xi\in\ell^2(X)$ such that
\begin{equation}\label{EQ:estimate in ONSL}
 \langle\xi,M_{\varphi_t}(T)\xi\rangle
 \ge \|M_{\varphi_t}(T)\|-\frac\varepsilon2\|T\|.
\end{equation}
Let $P_\xi$ be the rank-one projection onto $\CC\xi$ and define $D:=M_{\varphi_t}(P_\xi)$. Since $M_{\varphi_t}$ is completely positive, $D\ge0$. Moreover, we have $D_{xx}=|\xi_x|^2$, so $D$ is trace class and $\Tr(D)=1$. Also, we have
\[
 |D_{xy}|=|\varphi_t(x,y)|\cdot |\xi_x|\cdot |\xi_y|
 \le\omega(d(x,y))\sqrt{D_{xx}D_{yy}}.
\]
Since $\varphi_t$ is real and symmetric, direct calculations show that
\[
 \Tr(TD) = \Tr(T \cdot M_{\varphi_t}(P_\xi)) = \Tr(M_{\varphi_t}(T) \cdot P_\xi) = \langle\xi,M_{\varphi_t}(T)\xi\rangle.
\]
By \eqref{EQ:estimate in ONSL}, we obtain
\[
\langle\xi,M_{\varphi_t}(T)\xi\rangle \geq \|M_{\varphi_t}(T)\|-\frac\varepsilon2\|T\| \geq \|T\|-\|M_{\varphi_t}(T)-T\|-\frac\varepsilon2\|T\| \geq (1-\varepsilon)\|T\|.
\]
Thus, $\Tr(TD) \geq (1-\varepsilon)\|T\|$ and hence, $X$ satisfies ONSL.

Conversely, suppose that $X$ satisfies ONSL but does not coarsely embed into Hilbert space. Let $Y=\bigsqcup_nY_n\subseteq X$, $C, r_n, S_n, M_n, S_n^{(R)}$ be given by Proposition~\ref{prop:matrix-obstruction}. We regard $S_n$, $M_n$ and $S_n^{(R)}$ as operators on $\ell^2(X)$ by extending them by zero.
%The idea is to force the same density operator to have a uniformly positive pairing with $S_n$ and, at the same time, an asymptotically vanishing pairing with $S_n$. This will lead to a contradiction. The first conclusion comes from norm localisation applied to a finite-propagation truncation of $S_n$, while the second comes from the domination $S_n\le M_n$ and the distant support of $M_n$.

Choose $R>0$ such that $R^{-2} < \min\left\{1,\frac1{4C}\right\}$, and then choose $0<\varepsilon<1$ such that $2R^{-2}+6\varepsilon<\frac1{2C}$. Let $\omega$ be the function supplied by ONSL for this $R$ and $\varepsilon$. Regard
\[
 T_n^{(R)}:=S_n^{(R)}+3\Id_{\ell^2(X)}.
\]
By Proposition~\ref{prop:matrix-obstruction}, we have
\[
 \|S_n^{(R)}\|\le\|S_n\|+\|S_n-S_n^{(R)}\|\le 2+R^{-2}\leq 3.
\]
Hence, $T_n^{(R)}\ge 0$ and $\ppg(T_n^{(R)})\le R$. Moreover,
\[
 \lambda_{\max}(S_n^{(R)})
 \ge\lambda_{\max}(S_n)-R^{-2}
 \ge C^{-1}-R^{-2},
\]
which implies that
\[
3+C^{-1}-R^{-2} \le  \lambda_{\max}(S_n^{(R)}) + 3 = \lambda_{\max}(T_n^{(R)}) = \|T_n^{(R)}\| \leq \|S_n^{(R)}\| + 3  \leq 6.
\]

Choose a density operator $D_n$ supplied by ONSL for $T_n^{(R)}$ such that \eqref{EQ:TD control} and \eqref{EQ:D condition} hold for $\varepsilon,R$ witnessed by $\omega$. Then
\begin{equation}\label{eq:onsl-lower}
 \Tr(S_n^{(R)}D_n) =\Tr(T_n^{(R)}D_n)-3 \ge(1-\varepsilon)\|T_n^{(R)}\|-3 \ge C^{-1}-R^{-2}-6\varepsilon.
\end{equation}
Define the unit vector $v_n \in \ell^2(X)$ by $v_n(x):=\sqrt{(D_n)_{xx}}$. Since $S_n\le M_n$ and $D_n\ge0$, we have $\Tr(S_nD_n)\le\Tr(M_nD_n)$. Hence, by the condition on $M_n$, we have
\begin{align}\label{EQ:upper control on trace}
 \Tr(S_nD_n) &\leq \Tr(M_nD_n)\le\sum_{x,y\in Y_n}(M_n)_{xy} \cdot |(D_n)_{yx}|\le\omega(r_n)\sum_{x,y\in Y_n}(M_n)_{xy}v_n(x)v_n(y) \notag \\
 &=\omega(r_n)\langle v_n,M_nv_n\rangle \le\omega(r_n).
\end{align}
Furthermore, $\left|\Tr\left((S_n^{(R)}-S_n)D_n\right)\right| \le\|S_n^{(R)}-S_n\|\cdot\|D_n\|_1 \leq R^{-2}$.
Combining this with \eqref{eq:onsl-lower} and \eqref{EQ:upper control on trace}, we obtain
\[
 \omega(r_n) \geq \Tr(S_nD_n) \geq \Tr(S_n^{(R)}D_n) - R^{-2} \geq  C^{-1}-2R^{-2}-6\varepsilon
 >\frac1{2C}.
\]
This contradicts $\lim_{n\to \infty} \omega(r_n) = 0$, which finishes the proof.
\end{proof}

The terminology of ONSL is motivated by the operator norm localisation property (ONL), equivalent to the following: A discrete metric space $(X,d)$ of bounded geometry has ONL if for any $R>0$ and any $\varepsilon \in (0,1)$, there exists $S>0$ such that for any positive operator $T\in\B(\ell^2(X))$ with $\ppg(T)\le R$, there is a unit vector $v\in\ell^2(X)$ with $\diam(\supp v) \leq S$ and $\langle Tv, v \rangle \geq (1-\varepsilon)\|T\|$.

Moreover, we have the following connection. Intuitively speaking, the norm of $T$ is detected by a localised vector state in ONL, whereas ONSL allows a general normal state represented by a density operator, with localisation replaced by decay of its normalised correlations. This explains the terminology.

\begin{lem}\label{lem:ONL and rank one}
A discrete metric space $(X,d)$ of bounded geometry has ONL \emph{if and only if} $X$ has ONSL such that each $D$ can be chosen of rank one.
\end{lem}

\begin{proof}
Assume $X$ has ONL. Fixing $R>0$ and $\varepsilon \in (0,1)$, ONL provides $S>0$ such that for any $T \geq 0$ with $\ppg(T) \leq R$, there is a unit vector $v \in \ell^2(X)$ with $\diam(\supp v) \leq S$ and $\langle Tv,v\rangle \geq (1-\varepsilon)\|T\|$. Let $D=P_v$ be the rank-one projection in $\B(\ell^2(X))$ onto $\CC v$. Then $\Tr(TD) = \langle Tv,v\rangle \geq (1-\varepsilon) \|T\|$. On the other hand, define
\[
 \omega(t)=
 \begin{cases}
 1,&0\le t\le S,\\
 0,&t>S.
 \end{cases}
\]
Note that $D_{xy}=v_x\overline{v_y}$, which shows that $|D_{xy}|=\sqrt{D_{xx}D_{yy}}$. Then $D_{xx}D_{yy}\ne0$ implies $x,y\in\supp v$. In this case, we have $d(x,y)\le S$ and hence, $|D_{xy}|\le \omega(d(x,y))\sqrt{D_{xx}D_{yy}}$.

Conversely, fixing $R>0$ and $\varepsilon\in (0,1)$, let $\omega$ be supplied by ONSL. Choose $S>0$ with $\omega(S)<1$. For a positive $T$ with $\ppg(T)\le R$, denote the rank-one witness density operator by $D$. Since $\Tr(D)=1$ and $D$ is positive of rank one, $D=P_v$ the rank-one projection for some unit vector $v\in \ell^2(X)$. If $x,y\in\supp v$, then
\[
 \omega(d(x,y)) \geq  \frac{|D_{xy}|}{\sqrt{D_{xx}D_{yy}}}=1,
\]
which implies that $d(x,y) \leq S$. Hence, $\diam(\supp v) \leq S$, which proves ONL.
\end{proof}

%\subsection{Uniformly finite rank and property A}
Finally, we show a strengthened version of Lemma~\ref{lem:ONL and rank one}. Recall that $X$ has ONSL with uniformly finite rank if for each \(R>0\) and \(\varepsilon\in(0,1)\), the witnessing density operators $D$ in Definition~\ref{defn:ONSL} can be chosen with rank bounded by a constant \(k(R,\varepsilon)\).

\begin{prop}\label{prop:ONL and bounded rank}
A discrete metric space $(X,d)$ of bounded geometry has ONL \emph{if and only if} $X$ has ONSL with uniformly finite rank.
\end{prop}

\begin{proof}
Necessity follows directly from Lemma~\ref{lem:ONL and rank one}, so we only focus on sufficiency.
Fix $R>0$ and $\varepsilon \in (0,1)$. Applying ONSL with uniformly finite rank for $R$ and $\varepsilon$, we obtain $k \in \NN$ and a function $\omega$. Choose $S>0$ such that $\omega(S)<\frac1k$.

Let $T\in\B(\ell^2(X))$ be a positive operator with $\ppg(T)\le R$. Choose a density operator $D$ supplied by ONSL with rank at most $k$. Then there are vectors $v_x\in\CC^k$ such that $D_{xy}=\langle v_x,v_y\rangle_{\CC^k}$ for any $x,y\in X$. Set $Z:=\{x\in X:D_{xx}>0\}$, and define $u_x:=v_x/\|v_x\|$ for $x\in Z$. Then $\|u_x\|=1$ and we have
\[
 |\langle u_x,u_y\rangle|
 =\frac{|D_{xy}|}{\sqrt{D_{xx}D_{yy}}}
 \le\omega(d(x,y)).
\]
Claim: $Z$ cannot contain $k+1$ points whose pairwise distances are
greater than $S$.  Indeed, if $x_1,\ldots,x_{k+1}\in Z$ had this property,
then
\[
 |\langle u_{x_i},u_{x_j}\rangle| \leq  \omega(d(x_i,x_j)) \leq \omega(S) <\frac1k \quad \text{for} \quad i\neq j.
\]
Then the Gram matrix $G=\left(\langle u_{x_i},u_{x_j}\rangle\right)_{i,j=1}^{k+1}$ has diagonal entries equal to $1$, while in each row the sum of the absolute values of the off-diagonal entries is strictly less than $1$.  Hence $G$ is strictly diagonally dominant and therefore, invertible. This would make $u_{x_1},\ldots,u_{x_{k+1}}$ linearly independent in $\mathbb C^k$, which is a contradiction.

Choose a maximal family $z_1,\ldots,z_m\in Z$ with pairwise distances greater than $S$. The claim above gives $m\le k$, while maximality shows $Z\subseteq\bigcup_{j=1}^m B(z_j,S)$. By bounded geometry, we have
\begin{equation}\label{EQ:bdd geom}
 |Z|\le kN_X(S) <\infty.
\end{equation}
Since $D \geq 0$, we have $|D_{xy}|^2\le D_{xx}D_{yy}$. Hence $D_{xx}=0$ forces $D_{xy}=0$ for every $y$, which implies that $D= \chi_ZD\chi_Z$.

Decompose $Z=Z_1\sqcup\cdots\sqcup Z_q$ into its $R$-connected components, \emph{i.e.}, the connected components of the graph on \(Z\) where \(x\) and \(y\) are connected whenever \(d(x,y)\le R\). Then $d(Z_i,Z_j)>R$ for $i\neq j$.  Moreover, \eqref{EQ:bdd geom} shows that each $Z_i$ has at most $kN_X(S)$ points and hence,
\begin{equation}\label{EQ:diam of Zi}
 \diam(Z_i)\le (kN_X(S) -1)R.
\end{equation}
Let $T_i:=\chi_{Z_i}T\chi_{Z_i}$ and $D_i:=\chi_{Z_i}D\chi_{Z_i}$.
Since $\ppg(T)\le R$, we have $\chi_Z T \chi_Z = \sum_{i=1}^q T_i$. Then we obtain
\[
 \Tr(TD)=\Tr(\chi_ZT\chi_ZD\chi_Z) = \sum_{i=1}^q\Tr(T_iD_i).
\]
Since $T_i,D_i\ge0$, we have $\Tr(T_iD_i)\le\|T_i\|\Tr(D_i)$ and hence,
\[
 (1-\varepsilon)\|T\| \le\Tr(TD) \le\Big(\max_i\|T_i\|\Big)\sum_i\Tr(D_i) =\max_i\|T_i\|.
\]
Thus, there exists $i$ such that $\|T_i\|\ge(1-\varepsilon)\|T\|$.
Since $T_i$ is positive on the finite-dimensional space $\ell^2(Z_i)$, there
is a unit vector $v\in\ell^2(Z_i)$ such that
\[
 \langle Tv,v \rangle
 =\langle T_iv,v\rangle
 =\|T_i\|
 \ge(1-\varepsilon)\|T\|.
\]
By \eqref{EQ:diam of Zi}, we have $\diam(\supp v) \leq (kN_X(S) -1)R$, which concludes ONL.
\end{proof}

Combining with \cite[Theorem 4.1]{Sako14}, we obtain the following, which shows that the distinction between Property A and coarse
embeddability into Hilbert space can be detected by ranks of density operators used in ONSL.

\begin{cor}\label{cor:ONSL with bdd rank}
A discrete metric space of bounded geometry has Property A \emph{if and only if} it has ONSL with uniformly finite rank.
\end{cor}

\bibliographystyle{plain}
\bibliography{reference_CE_submission}

\end{document}